\RequirePackage{fix-cm}
\documentclass[smallextended]{svjour3}       
\smartqed  
\usepackage{float}
\usepackage{cite}%
\usepackage{color}
\usepackage{graphicx}
\usepackage{amscd, amsmath, amssymb}
\usepackage{pgf,tikz}
\usetikzlibrary{arrows}
\usepackage{pgfplots}
\usepackage{enumerate}
\usepackage{caption}
\usepackage{subcaption}
\usepackage{tikz}
\usepackage{pgfplots}
\usepackage{adjustbox}
\usepackage{hyperref}
\usepackage{todonotes}
\usepackage{amsmath}
\usepackage{amssymb}

\usepackage{amsthm}
\usepackage{amsfonts}
\usepackage{enumerate}
\usepackage{mathtools} 
\usepackage{xcolor}
\usepackage{rotating,booktabs,multirow}
\usepackage{varwidth}
\usepackage{algorithm,algpseudocode}
\algnewcommand{\Inputs}[1]{%
  \State \textbf{Inputs:}
  \Statex \hspace*{\algorithmicindent}\parbox[t]{.8\linewidth}{\raggedright #1}
}
\algnewcommand{\Initialize}[1]{%
  \State \textbf{Initialization:}
  \Statex \hspace*{\algorithmicindent}\parbox[t]{.8\linewidth}{\raggedright #1}
}

\DeclareMathOperator{\argmin}{argmin}
\DeclareMathOperator{\st}{s.t.}
\DeclareMathOperator{\tr}{tr}

\usepackage{algpseudocode}

\algnewcommand{\algorithmicand}{\textbf{ and }}
\algnewcommand{\algorithmicor}{\textbf{ or }}
\algnewcommand{\OR}{\algorithmicor}
\algnewcommand{\AND}{\algorithmicand}
\algnewcommand{\var}{\texttt}
\newcommand\scalemath[2]{\scalebox{#1}{\mbox{\ensuremath{\displaystyle #2}}}}
\begin{document}

\title{The exact worst-case convergence rate of the alternating direction method of multipliers\thanks{This work was supported by the
 Dutch Scientific Council (NWO)  grant OCENW.GROOT.2019.015, \emph{Optimization for and with Machine Learning (OPTIMAL)}.}}

\titlerunning{Convergence rate of ADMM}        

\author{ Moslem Zamani
   \and Hadi Abbaszadehpeivasti
     \and
   Etienne de Klerk}


\institute{
M. Zamani\at Tilburg University, Department of Econometrics and Operations Research, Tilburg, The Netherlands\\
\email{m.zamani\_1@tilburguniversity.edu}
\and
H. Abbaszadehpeivasti\at Tilburg University, Department of Econometrics and Operations Research, Tilburg, The Netherlands\\
\email{h.abbaszadehpeivasti@tilburguniversity.edu}
\and
E. de Klerk\at Tilburg University, Department of Econometrics and Operations Research, Tilburg, The Netherlands\\
\email{e.deklerk@tilburguniversity.edu}
}

\date{Received: date / Accepted: date}

\maketitle

\begin{abstract}
Recently, semidefinite programming performance estimation has been employed as a strong tool for the worst-case performance analysis of first order methods. In this paper, we derive new non-ergodic convergence rates for the alternating direction method of multipliers (ADMM) by using performance estimation. We give some examples which show the exactness of the given bounds. We also study the linear and R-linear convergence of ADMM. We establish that ADMM enjoys a global linear convergence rate if and only if the dual objective satisfies the  Polyak-\L ojasiewicz (P\L) inequality in the presence of strong convexity. In addition, we give an explicit formula for the linear convergence rate factor. Moreover, we study the R-linear convergence of ADMM under two new scenarios.
\end{abstract}

\section{Introduction}\label{intro}

We consider the optimization problem
\begin{align}\label{P}
\min_{(x, z)\in \mathbb{R}^n\times\mathbb{R}^m} f(x)+g(z),\\
\nonumber \st \ Ax+Bz=b,
\end{align}
where $f: \mathbb{R}^n\to \mathbb{R}\cup \{\infty\}$ and $g: \mathbb{R}^m\to \mathbb{R}\cup \{\infty\}$ are closed proper convex functions, $0\neq A\in\mathbb{R}^{r\times n}$, $0\neq B\in\mathbb{R}^{r\times m}$ and $b\in\mathbb{R}^{r}$. Moreover, we assume that $(x^\star, z^\star)$ is an optimal solution of problem \eqref{P} and $\lambda^\star$ is its corresponding Lagrange multipliers. Moreover, we denote the value of $f$ and $g$ at $x^\star$ and $z^\star$ with $f^\star$ and $g^\star$, respectively.

Problem \eqref{P} appears naturally (or after variable splitting) in many applications in statistics, machine learning and image processing to name but a few \cite{boyd2011, rudin1992nonlinear, hastie2015statistical, linalternating}. The most common method for solving problem \eqref{P} is the alternating direction method of multipliers (ADMM).  ADMM is dual based approach that exploits separable structure and it may be described as follows.
\begin{algorithm}
\caption{ADMM}
\begin{algorithmic}
\State Set $N$ and $t>0$ (step length), pick $\lambda^0, z^0$.
\State For $k=1, 2, \ldots, N$ perform the following step:\\
\begin{enumerate}
\item
$x^{k}\in\argmin f(x)+\langle\lambda^{k-1},Ax\rangle+\tfrac{t}{2}\|Ax+Bz^{k-1}-b\|^2$
\item
$z^{k}\in\argmin g(z)+\langle\lambda^{k-1},Bz\rangle+\tfrac{t}{2}\|Ax^k+Bz-b\|^2$
\item
$\lambda^k=\lambda^{k-1}+t(Ax^k+Bz^k-b)$.
\end{enumerate}
\end{algorithmic}
\label{ADMM}
\end{algorithm}

ADMM was first proposed in \cite{gabay1976dual, glowinski1975approximation} for solving nonlinear variational problems. We refer the interested reader to \cite{glowinski2017splitting} for a historical review of ADMM. The popularity of ADMM is due to its capability to be implemented  parallelly and hence can handle large-scale problems  \cite{boyd2011, han2022survey, madani2015admm, stellato2020osqp}. For example, it is used for solving inverse problems governed by partial differential equation forward models \cite{lozenski2021consensus}, and distributed energy resource coordinations \cite{liu2022distributed}, to mention but a few.

 The convergence of ADMM has been investigated extensively in the literature and there exist many convergence results. However, different performance measures have been used for the computation of convergence rate; see \cite{franca2018admm, goldstein2014fast, sabach2022faster, goldfarb2013fast, he20121, monteiro2013iteration, linalternating, li2019accelerated}. In this paper, we consider the dual objective value as a performance measure.

Throughout the paper, we assume that each subproblem in steps 1 and 2 of Algorithm \ref{ADMM} attains its minimum.  The Lagrangian function of problem \eqref{P} may be written as
\begin{align}\label{lag}
L(x, z, \lambda)= f(x)+g(z)+\langle \lambda, Ax+Bz-b\rangle,
\end{align}
and the dual objective of problem \eqref{P} is also defined as
\begin{align*}
\nonumber D(\lambda)&=\min_{(x, z)\in \mathbb{R}^n\times\mathbb{R}^m} f(x)+g(z)+\langle \lambda, Ax+Bz-b\rangle.
\end{align*}
 We assume throughout the paper that strong duality holds for problem \eqref{P}, that is
$$
\max_{\lambda\in\mathbb{R}^r} D(\lambda)=\min_{Ax+Bz=b} f(x)+g(z).
$$
Note that we have strong duality when both functions $f$ and $g$ are real-valued. For extended convex functions, strong duality holds under some mild conditions; see e.g. \cite[Chapter 15]{beck2017first}.

Some common performance measures for the analysis of ADMM are as follows,
\begin{itemize}
\item
Objective value: $\left|f(x^N)+g(z^N)-f^\star-g^\star\right|$;
\item
Primal and dual feasibility: $\left\|Ax^N+Bz^N-b\right\|$ and $\left\|A^TB(z^N-z^{N-1})\right\|$;
\item
Dual objective value: $D(\lambda^\star)-D(\lambda^N)$;
\item
Distance between $(x^N, z^N, \lambda^N)$ and a saddle points of problem \eqref{lag}.
\end{itemize}
Note that the mathematical expressions are written in a non-ergodic sense for convenience.  Each measure is useful in monitoring the progress and convergence of ADMM. The objective value is the most commonly used performance measure for the analysis of algorithms in convex optimization \cite{bertsekas2015convex, beck2017first, Nesterov}. As mentioned earlier, ADMM is a dual based method and it may be interpreted as a proximal method applied to the dual problem; see \cite{bertsekas2015convex, linalternating} for further discussions and insights. Thus, a natural performance measure for ADMM would be dual objective value.  In this study, we investigate the convergence rate of ADMM in terms of  dual objective value and feasibility. It worth noting that most performance measures may be analyzed through the framework developed in Section \ref{sec_pep}.

Regarding dual objective value, the following convergence rate is known in the literature. This theorem holds for strongly convex functions $f$ and $g$; recall that $f$ is called strongly convex with modulus $\mu\geq0$ if the function $f-\tfrac{\mu}{2}\|\cdot\|^2$ is convex.
 \begin{theorem}\cite[Theorem 1]{goldstein2014fast}\label{TT1}
Let $f$ and $g$ be strongly convex with moduli $\mu_1>0$ and $\mu_2>0$, respectively. If $t\leq \sqrt[3]{\frac{\mu_1\mu_2^2}{\lambda_{\max} (A^TA)\lambda_{\max}^2 (B^TB)}}$, then
\begin{align}\label{B1}
D(\lambda^\star)-D(\lambda^N)\leq \frac{\|\lambda^1-\lambda^\star\|^2}{2t(N-1)}.
\end{align}
\end{theorem}

In this study we establish that Algorithm \ref{ADMM} has the convergence rate of $O(\tfrac{1}{N})$ in terms of dual objective value without assuming the strong convexity of $g$. Under this setting, we also prove that Algorithm \ref{ADMM} has the convergence rate of $O(\tfrac{1}{N})$ in terms of primal and dual residuals. Moreover, we  show that the given bounds are exact. Furthermore, we study the linear and R-linear convergence.

\subsubsection*{Outline of our paper}
Our paper is structured as follows. We present the semidefinite programming (SDP) performance estimation method in Section \ref{sec_pep}, and we develop the performance estimation to handle dual based methods including ADMM. In Section \ref{sec_f}, we derive some new non-asymptotic convergence rates by using performance estimation for ADMM in terms of dual function, primal and dual residuals. Furthermore, we show that the given bounds are tight by providing some examples. In Section \ref{sec_L} we proceed with the study of the linear convergence of ADMM. We establish that ADMM enjoys a linear convergence if and only if the dual function satisfies the P\L\ inequality when the objective function is strongly convex. Furthermore, we investigate the relation between the P\L\ inequality and common conditions used by scholars to prove the linear convergence. Section \ref{Sec.R} is devoted to the R-linear convergence. We prove that ADMM is  R-linear convergent under two new scenarios which are weaker than the existing ones in the literature.
\subsubsection*{Terminology and notation}
In this subsection we review some definitions and concepts from convex analysis. The interested reader is referred to the classical text by Rockafellar \cite{rockafellar2015convex} for more information.  The $n$-dimensional Euclidean space is denoted by $\mathbb{R}^n$.
 We use $\langle \cdot, \cdot\rangle$ and $\| \cdot\|$ to denote the Euclidean inner product and norm, respectively. The column vector $e_i$ represents the $i$-th standard unit vector and $I$ stands for the identity matrix.
  For a matrix $A$, $A_{i, j}$ denotes its $(i, j)$-th entry,
  and $A^T$ represents the transpose of $A$. The  notation $A\succeq 0$ means the matrix $A$ is symmetric positive semidefinite. We use $\lambda_{\max} (A)$ and $\lambda_{\min} (A)$ to denote  the largest and the smallest eigenvalue of symmetric matrix $A$, respectively. Moreover, the seminorm $\| \cdot\|_A$ is defined as $\|x\|_A=\|Ax\|$ for any $A\in\mathbb{R}^{m\times n}$.

 Suppose that $f:\mathbb{R}^n\to(-\infty, \infty]$ is an extended convex function.
The function $f$ is called closed if its epi-graph is closed,
 that is $\{(x, r): f(x)\leq r\}$ is a closed subset of $\mathbb{R}^{n+1}$. The function $f$ is said to be proper if there exists $x\in\mathbb{R}^n$ with $f(x)<\infty$.
 We denote the set of  proper and closed convex functions  on $\mathbb{R}^n$ by $\mathcal{F}_{0}(\mathbb{R}^n)$.
 The subgradients of $f$ at $x$ is denoted and defined as
$$
\partial f(x)=\{\xi: f(y)\geq f(x)+\langle \xi, y-x\rangle, \forall y\in\mathbb{R}^n\}.
$$
We call a differentiable function $f$ $L$-smooth if for any $x_1, x_2\in\mathbb{R}^n$,
$$
\left\|\nabla f(x_1)-\nabla f(x_2) \right\|\leq L\|x_1-x_2\| \ \ \forall x_1, x_2\in\mathbb{R}^n.
$$

\begin{definition}
Let $f:\mathbb{R}^n\to(-\infty, \infty]$ be a closed proper function and let $A\in\mathbb{R}^{m\times n}$. We say $f$ is $c$-strongly convex relative to $\| .\|_A$ if the function
$f-\tfrac{c}{2} \| . \|_A^2$ is convex.
\end{definition}

 In the rest of the section, we assume that $A\in\mathbb{R}^{m\times n}$. It is seen that any $\mu$-strongly convex function is $\tfrac{\mu}{\lambda_{\max}(A^TA)}$-strongly convex relative to $\| .\|_A$. However, its converse does not necessarily hold unless $A$ has full column rank. Hence, the assumption of strong convexity relative to $\| .\|_A$ for a given matrix $A$ is weaker compared to the assumption of strong convexity. For further details on the strong convexity in relation to a given function, we refer the reader to \cite{lu2018relatively}.
We denote the set of  $c$-strongly convex functions relative to $\| .\|_A$ on $\mathbb{R}^n$ by $\mathcal{F}_{c}^A(\mathbb{R}^n)$. We denote the distance function to the set $X$ by $d_X(x):=\inf_{y\in X}\|y-x\|$.

In the following sections we derive some new convergence rates for ADMM by using performance estimation. The main idea of  performance estimation is based on  interpolablity.
Let $\mathcal{I}$ be an index set and let $\{(x^i; g^i; f^i)\}_{i\in \mathcal{I}}\subseteq \mathbb{R}^n\times \mathbb{R}^n\times \mathbb{R}$.
A set $\{(x^i; \xi^i; f^i)\}_{i\in \mathcal{I}}$ is called $\mathcal{F}^A_{c}$-interpolable if there exists $f\in\mathcal{F}^A_{c}(\mathbb{R}^n)$
 with
$$
f(x^i)=f^i, \ \xi^i\in\partial f(x^i) \ \ i\in\mathcal{I}.
$$
The next theorem gives necessary and sufficient conditions for $\mathcal{F}_{c}^A$-interpolablity.
\begin{theorem}\label{T1}
Let $c\in [0, \infty)$ and let $\mathcal{I}$ be an index set.
The set $\{(x^i; \xi^i; f^i)\}_{i\in \mathcal{I}}\subseteq \mathbb{R}^n\times \mathbb{R}^n \times \mathbb{R}$ is $\mathcal{F}^A_{c}$-interpolable if and only if for any $i, j\in\mathcal{I}$, we have
\begin{align}\label{interp}
\tfrac{c}{2}\left\|x^i-x^j\right\|_A^2\leq f^i-f^j-\left\langle \xi^j, x^i-x^j\right\rangle.
\end{align}
Moreover, $\mathcal{F}_{0}$-interpolable and $L$-smooth if and only if for any $i, j\in\mathcal{I}$, we have
\begin{align}
\tfrac{1}{2L}\left\|g^i-g^j\right\|^2\leq f^i-f^j-\left\langle g^j, x^i-x^j\right\rangle.
\end{align}
\end{theorem}
\begin{proof}
 The argument is analogous to that of \cite[Theorem 4]{taylor2017smooth}. The triple $\{(x^i; \xi^i; f^i)\}_{i\in \mathcal{I}}$ is $\mathcal{F}^A_{c}$-interpolable if and only if
 the triple $\{(x^i; \xi^i-cA^TAx^i; f^i-\tfrac{c}{2}\|x^i\|_A^2)\}_{i\in \mathcal{I}}$ is $\mathcal{F}_{0}$-interpolable. By \cite[Theorem 1]{taylor2017smooth},  $\{(x^i; \xi^i-cA^TAx^i; f^i-\tfrac{c}{2}\|x^i\|_A^2)\}_{i\in \mathcal{I}}$ is $\mathcal{F}_{0}$-interpolable if and only if
 $$
  f^i-\tfrac{c}{2}\left\|x^i \right\|_A^2\geq  f^j-\tfrac{c}{2}\left\|x^j \right\|_A^2-\left\langle \xi^j-cA^TAx^j, x^i-x^j\right\rangle
 $$
 which implies inequality \eqref{interp}. The second part  follows directly from \cite[Theorem 4]{taylor2017smooth}.
\end{proof}
Note  that any convex function is $0$-strongly convex relative to $A$. Let $f\in\mathcal{F}_{0}(\mathbb{R}^n)$. The conjugate function $f^*:\mathbb{R}^n\to (-\infty, \infty]$ is defined as $f^*(y)=\sup_{x\in\mathbb{R}^n} \langle y, x\rangle-f(x)$. We have the following identity
\begin{align}\label{Conj_inv}
  \xi\in\partial f(x) \  \ \Leftrightarrow \ \ x\in\partial f^*(\xi).
\end{align}
Let $f\in\mathcal{F}_{0}(\mathbb{R}^n)$ be $\mu$-strongly convex. The function $f$ is $\mu$-strongly convex if and only if $f^*$ is $\tfrac{1}{\mu}$-smooth. Moreover, $(f^*)^*=f$.

By using conjugate functions, the dual of problem \eqref{P} may be written as
\begin{align}\label{Con_C}
\nonumber D(\lambda)&=\min_{(x, z)\in \mathbb{R}^n\times\mathbb{R}^m} f(x)+g(z)+\langle \lambda, Ax+Bz-b\rangle\\
& =-\langle \lambda, b\rangle-f^*(-A^T\lambda)-g^*(-B^T\lambda).
\end{align}
By the optimality conditions for the dual problem, we get
\begin{align}\label{Con_opt_D}
b-Ax^\star-Bz^\star=0,
\end{align}
for some $ x^\star\in\partial f^*(-A^T\lambda^\star)$ and $ z^\star\in\partial g^*(-B^T\lambda^\star)$. Equation \eqref{Con_opt_D} with \eqref{Conj_inv} imply that $( x^\star, z^\star)$ is an optimal solution to problem \eqref{P}.

The optimality conditions for the subproblems of Algorithm \ref{ADMM} may be written as
\begin{align}\label{OPT}
\nonumber & 0\in \partial f(x^k)+A^T\lambda^{k-1}+tA^T\left(Ax^k+Bz^{k-1}-b\right),\\
& 0\in \partial g(z^k)+B^T\lambda^{k-1}+tB^T\left(Ax^k+Bz^{k}-b\right).
\end{align}
As  $\lambda^k=\lambda^{k-1}+t(Ax^k+Bz^k-b)$, we get
\begin{align}\label{OPT_R}
0\in \partial f(x^k)+A^T\lambda^k+tA^TB\left(z^{k-1}-z^k\right), \ \ 0\in \partial g(z^k)+B^T\lambda^k.
\end{align}
So, $(x^k, z^k)$ is optimal for dual objective at $\lambda ^k$ if and only if $A^TB\left(z^{k-1}-z^k\right)=0$. We call $A^TB\left(z^{k-1}-z^k\right)$ dual residual.

\section{Performance estimation}\label{sec_pep}
In this section, we develop the performance estimation for ADMM. The performance estimation method introduced by Drori and Teboulle \cite{drori2014performance} is an SDP-based method for the analysis of first order methods. Since then, many scholars employed this strong tool to derive the worst case convergence rate of different iterative methods; see \cite{ryu2020operator, kim2016optimized, abbaszadehpeivasti2021rate, taylor2017smooth} and the references therein. Moreover, Gu and Yang \cite{gu2020dual} employed performance estimation to study the extension of the dual step length for ADMM. Note that while there are some similarities between our work and \cite{gu2020dual} in using performance estimation, the formulations and results are different.

The worst-case convergence rate of Algorithm \ref{ADMM} with respect to dual objective value may be cast as the following abstract optimization problem,
\begin{align}\label{P1}
\nonumber   \max & \ D(\lambda^\star)-D(\lambda^N)\\
\nonumber \st &  \  \{x^k, z^k, \lambda^k\}_1^N \ \textrm{is generated  by Algorithm \ref{ADMM} w.r.t.}\ f, g, A, B, b, \lambda^0,  z^0, t  \\
\nonumber  & \ (x^\star, z^\star) \ \text{is an optimal solution with Lagrangian multipliers }\ \lambda^\star\\
\nonumber  & \  \|\lambda^0-\lambda^\star \|^2+t^2\left\| z^0-z^\star\right\|_B^2=\Delta\\
  & \ f\in \mathcal{F}^A_{c_1}(\mathbb{R}^n), g\in \mathcal{F}^B_{c_2}(\mathbb{R}^m)\\
\nonumber  & \ \lambda^0\in\mathbb{R}^r,  z^0\in\mathbb{R}^m, A\in\mathbb{R}^{r\times n}, B\in\mathbb{R}^{r\times m},b\in\mathbb{R}^{r},
\end{align}
where $f, g, A, B, b,  z^0, \lambda^0, x^\star, z^\star, \lambda^\star$ are decision variables and $N, t, c_1, c_2, \Delta$ are the given parameters.

By using Theorem \ref{T1} and the optimality conditions \eqref{OPT}, problem \eqref{P1} may be reformulated as the finite dimensional optimization problem,
\begin{align}\label{P2}
\nonumber   \max & \ D(\lambda^\star)-D(\lambda^N)\\
\nonumber \st &  \  \{(x^k; \xi^k; f^k)\}_1^{N}\cup\{(x^\star; \xi^\star; f^\star)\} \ \textrm{satisfy interpolation constraints \eqref{interp}}  \\
\nonumber & \{(z^k;  \eta^k; g^k)\}_0^N\cup\{(z^\star; \eta^\star; g^\star)\} \ \textrm{satisfy interpolation constraints \eqref{interp}}  \\
\nonumber  & \ (x^\star, z^\star) \ \text{is an optimal solution with Lagrangian multipliers }\ \lambda^\star\\
  & \  \|\lambda^0-\lambda^\star \|^2+t^2\left\| z^0-z^\star\right\|_B^2= \Delta\\
\nonumber  & \ \xi^k=tA^Tb-tA^TAx^k-tA^TBz^{k-1}-A^T\lambda^{k-1}, \ \ k\in\{1, ..., N\}\\
\nonumber  & \ \eta^k=tB^Tb-tB^TAx^k-tB^TBz^{k}-B^T\lambda^{k-1}, \ \ k\in\{1, ..., N\}\\
\nonumber  & \ \lambda^k=\lambda^{k-1}+t(Ax^k+Bz^{k}-b), \ \ k\in\{1, ..., N\}\\
\nonumber  & \ \lambda^0\in\mathbb{R}^r,  z^0\in\mathbb{R}^m, A\in\mathbb{R}^{r\times n}, B\in\mathbb{R}^{r\times m},b\in\mathbb{R}^{r}.
\end{align}

To handle problem \eqref{P2}, without loss of generality, we assume that the matrix $\begin{pmatrix} A & B \end{pmatrix}$ has full row rank. Note this assumption does not employed in our arguments in the following sections. In addition, we introduce some new variables. As problem \eqref{P} is invariant under translation of $(x,z)$, we may assume without loss of generality that $b=0$ and $(x^\star, z^\star)=(0, 0)$.  In addition, due to the full row rank of the matrix $\begin{pmatrix} A & B \end{pmatrix}$, we may assume that $\lambda^0=\begin{pmatrix} A & B \end{pmatrix} \begin{pmatrix} x^\dag \\   z^\dag \end{pmatrix}$ and $\lambda^\star=\begin{pmatrix} A & B \end{pmatrix} \begin{pmatrix} \bar x \\ \bar z \end{pmatrix}$ for some $\bar x,  x^\dag, \bar z,  z^\dag$. So,
$$
\xi^\star=-A^TA\bar x-A^TB\bar z\in \partial f(0), \ \ \eta^\star=-B^TA\bar x-B^TB\bar z\in \partial g(0),
$$
and $D(\lambda^\star)=f^\star+g^\star$.

By using equality constraints of problem \eqref{P2} and the newly introduced  variables, we have for $k\in\{1, ..., N\}$

{\small{
\begin{align}\label{Eqq}
 & \lambda^k=(Ax^\dag+Bz^\dag)+\sum_{i=1}^{k} t(Ax^i+Bz^i),\\
\nonumber & -(A^TAx^\dag+A^TBz^\dag)-\sum_{i=1}^{k-1} t(A^TAx^i+A^TBz^i)-tA^TAx^k-tA^TBz^{k-1}\in \partial f(x^k),\\
\nonumber & -(B^TAx^\dag+B^TBz^\dag)-\sum_{i=1}^{k} t(B^TAx^i+B^TBz^i)\in\partial g(z^k).
\end{align}
}}

Note that $(\tilde x, \tilde z)\in\argmin f(x)+g(z)+\langle \lambda^N, Ax+Bz-b\rangle$ if and only if
\begin{align}\label{opt_lam_N}
0\in \partial f(\tilde x)+A^T\lambda^N, \ \ \  0\in \partial g(\tilde z)+B^T\lambda^N.
\end{align}
It is worth noting that a point $\tilde{x}$ satisfying these conditions exists, as  function $f$ is strongly convex relative to $A$. In addition, one may consider $\tilde z=z^N$ by virtue of \eqref{OPT_R}.  For the sake of notation convenience, we introduce $x^{N+1}=\tilde x$. The reader should bear in mind that $x^{N+1}$ is not generated by Algorithm \ref{ADMM}. Therefore,  $D(\lambda^N)=f(x^{N+1})+g(z^N)+\left\langle \lambda ^N, Ax^{N+1}+Bz^N\right\rangle$ for some $x^{N+1}$ with
 $-A^T\lambda^N\in\partial f(x^{N+1})$.
Hence, problem \eqref{P2} may be written as

{\small{
\begin{align}\label{P3}
\nonumber   \max \ \ & f^\star+g^\star-f^{N+1}-g^{N}-\left\langle Ax^\dag+Bz^\dag+\sum_{i=1}^{N} t(Ax^i+Bz^i), Ax^{N+1}+Bz^{N}\right\rangle \\
\nonumber  \st \ \  &  \  \tfrac{c_1}{2}\left\|x^k-x^j\right\|_A^2\leq \left\langle Ax^\dag+Bz^\dag+\sum_{i=1}^{k-1} t(Ax^i+Bz^i)+tAx^k+tBz^{k-1}, A(x^j-x^k) \right\rangle+\\
\nonumber  &  \  \  \  \ \  \  \  \ f^j-f^k, \ \ k\in\{1, \dots, N\},\ \ \ j\in\{1, \dots, N+1\},\\
\nonumber  &  \  \tfrac{c_1}{2}\left\|x^{N+1}-x^j\right\|_A^2\leq \left\langle Ax^\dag+Bz^\dag+\sum_{i=1}^{N} t(Ax^i+Bz^i), A\left(x^j-x^{N+1}\right) \right\rangle+\\
\nonumber  &  \  \  \  \ \  \  \  \  f^j-f^{N+1}, \ \ \  j\in\{1, \dots, N\},\\
\nonumber  &  \  \tfrac{c_2}{2}\left\|z^k-z^j\right\|_B^2\leq \left\langle Ax^\dag+Bz^\dag+\sum_{i=1}^{k} t(Ax^i+Bz^i), B\left(z^j-z^k\right) \right\rangle+\\
  &  \  \  \  \ \  \  \  \ g^j-g^k, \ \ \  j, k\in\{1, \dots, N\},\\
\nonumber &  \  \tfrac{c_1}{2}\left\|x^k\right\|_A^2\leq f^k-f^\star+\left\langle A\bar x+B\bar z, Ax^k \right\rangle, \ \ \  k\in\{1, \dots, N+1\},\\
\nonumber  &  \  \tfrac{c_1}{2}\left\|x^k\right\|_A^2\leq -\left\langle Ax^\dag+Bz^\dag+\sum_{i=1}^{k-1} t(Ax^i+Bz^i)+tAx^k+tBz^{k-1}, Ax^k \right\rangle+\\
\nonumber  &    \  \  \  \ \  \  \  \ f^\star-f^k, \ \ \   k\in\{1, \dots, N\},\\
\nonumber  &  \  \tfrac{c_1}{2}\left\|x^{N+1}\right\|_A^2\leq f^\star-f^{N+1}- \left\langle Ax^\dag+Bz^\dag+\sum_{i=1}^{N} t(Ax^i+Bz^i), Ax^{N+1} \right\rangle,\\
\nonumber &  \  \tfrac{c_2}{2}\left\|z^k\right\|_B^2\leq g^k-g^\star+\left\langle A\bar x+B\bar z, Bz^k \right\rangle, \ \ \  k\in\{1, \dots, N\},\\
\nonumber  &  \  \tfrac{c_2}{2}\left\|z^k\right\|_B^2\leq g^\star-g^k-\left\langle Ax^\dag+Bz^\dag+\sum_{i=1}^{k} t(Ax^i+Bz^i), Bz^k \right\rangle, \ \ \  k\in\{1, \dots, N\},\\
\nonumber  & \  \left\|Ax^\dag+Bz^\dag-(A\bar x+B\bar z)\right\|^2+t^2\left\| z^0\right\|_B^2= \Delta,\\
\nonumber  & \ x^\dag\in\mathbb{R}^n,  z^0, z^\dag\in\mathbb{R}^m, A\in\mathbb{R}^{r\times n}, B\in\mathbb{R}^{r\times m}.
\end{align}
}}

In problem \eqref{P3}, $A, B, \{x^k, f^k\}_1^{N+1}, \{z^k, g^k\}_1^{N}, x^\dag, z^\dag, \bar x, f^\star, \bar z, g^\star,  z^0$ are decision variables. By using the Gram matrix method, problem \eqref{P3} may be relaxed as a semidefinite program as follows. Let
\begin{align*}
&U=\begin{pmatrix}
    x^\dag & x^1 & \dots & x^{N+1} & \bar x
  \end{pmatrix},\ \ \
 & V=\begin{pmatrix}
    z^\dag & z^0 & \dots & z^{N} & \bar z
  \end{pmatrix}.
  \end{align*}
  By introducing matrix variable
  \begin{align*}
 Y=\begin{pmatrix}
    AU  & BV
  \end{pmatrix}^T
  \begin{pmatrix}
    AU  & BV
  \end{pmatrix},
\end{align*}
problem \eqref{P3} may be relaxed as the following SDP,
\begin{align}\label{P4}
\nonumber   \max & \ f^\star+g^\star-f^{N+1}-g^{N}-\tr(L_oY) \\
\nonumber \st &  \ \tr(L_{i, j}^fY)\leq f^i-f^j, \ \ i,j\in\{1, ..., N+1, \star\}  \\
\nonumber &  \ \tr(L_{i, j}^gY)\leq g^i-g^j ,  \ \ i,j\in\{1, ..., N, \star\}  \\
 & \  \tr(L_0 Y)= \Delta\\
\nonumber & \   Y\succeq 0,
\end{align}
where the constant matrices $L_{i, j}^f, L_{i, j}^g, L_o, L_0$ are determined according to the constraints of problem \eqref{P3}. In the following sections, we present some new convergence results that are derived by solving  this kind of formulation.
\section{Worst-case convergence rate}\label{sec_f}

In this section, we provide  new convergence rates for ADMM with respect to some performance measures. Before we get to the theorems we need to present some lemmas.

\begin{lemma}\label{Lemma1}
Let $N\geq 4$ and $t, c \in\mathbb{R}$. Let $E(t, c )$ be $(N+1)\times (N+1)$ symmetric matrix given by
\[
  E(t, c )=\scalemath{1}{\left(\begin{array}{cccccccccc}
      2c  & 0 & 0  &  0       & \dots & 0  & 0 & \dots & 0 & t-c \\
      0 & \alpha_2 & \beta_2  &  0       & \dots & 0  & 0 & \dots & 0 & -t\\
      0 & \beta_2  & \alpha_3 & \beta_3  & \dots & 0  & 0 &  \dots & 0 & t\\
       \vdots &\vdots & \vdots & \vdots  & \vdots & \vdots & \vdots& \vdots & \vdots & \vdots\\
        0 & 0 & 0 & 0 & \dots & \alpha_k & \beta_k & \dots & 0 & t\\
        \vdots & \vdots & \vdots & \vdots  & \vdots & \vdots & \vdots& \vdots & \vdots & \vdots\\
         0 & 0 & 0  &  0       & \dots & 0  & 0 & \dots & \alpha_{N} & \beta_{N}\\
         t-c  & -t & t & t & \dots & t  & t &  \dots & \beta_{N} &  \alpha_{N+1}\\
\end{array}\right)},
\]
where
{\small{
\begin{align*}
\alpha_k=
\begin{cases}
  6c -5t, & {k=2}\\
 2\left(2k^2-3k+1\right)c -\left(4k-1\right)t, & {3\leq k\leq N-1}\\
  2N(N-1)c -(2N+1)t, & {k=N} \\
  2Nc -(N+1)t, & {k=N+1},
\end{cases}
\\
\beta_k=
\begin{cases}
  2kt-(2k^2-k-1)c , & {2\leq k\leq N-1 } \\
  3t-2(N-1)c , & {k=N},
\end{cases}
\end{align*}
}}
 and $k$ denotes row number. If $c >0$ is given, then
$$
[0, c ]\subseteq \{t: E(t, c )\succeq 0\}.
$$
\end{lemma}
\begin{proof}
As $\{t: E(t, c )\succeq 0\}$ is  a convex set, it suffices to prove the  positive semidefiniteness of $E(0, c )$ and $E(c , c )$. Since $E(0, c )$ is diagonally dominant, it is positive semidefinite. Now, we establish that the matrix $K=E(1,1)$
is positive definite. To this end, we show that all leading principal minors of $K$ are positive. To compute the leading principal minors, we perform the following elementary row operations on $K$:
\begin{enumerate}[i)]
  \item
   Add the second row to the third row;
   \item
    Add the second row to the last row;
  \item
  Add the third row to the forth row;
  \item
  For $i=4:N-1$
  \begin{itemize}
    \item
     Add $i-th$ row to $(i+1)-th$ row;
    \item
     Add $\tfrac{3-i}{2i^2-3i-1}$ times of $i-th$ row to the last row;
  \end{itemize}
  \item
  Add $\frac{N-1}{3N-5}$ times of $N-th$ row to $(N+1)-th$ row.
\end{enumerate}
It is seen that $K_{k-1, k}+K_{k, k}=-K_{k+1, k}$ for $2\leq k\leq N-1$. Hence, by performing these operations, we get an upper triangular matrix $J$ with diagonal
\[
J_{k, k}=\begin{cases}
  2, & {k=1}\\
  2k^2-3k-1, & {2\leq k\leq N-1}\\
 3N-5, & {k=N} \\
N-2-\frac{(N-1)^2}{ 3N-5}-\sum_{i=4}^{N-1}\tfrac{(i-3)^2}{2i^2-3i-1} , & {k=N+1}. \end{cases}
\]
 It is seen all first $N$ diagonal elements of $J$ are positive. We show that $J_{N+1, N+1}$ is also positive.  For $i\geq 4$ we have
  \begin{align}\label{Lem_inq_pos}
  \tfrac{(i-3)^2}{2i^2-3i-1}\leq \tfrac{(i-1)^2+4}{2(i-1)^2}\leq \tfrac{1}{2}+\tfrac{2}{(i-1)(i-2)}.
  \end{align}
So,
  $$
 \tfrac{2N^2-9N+9}{3N-5}-\sum_{i=4}^{N-1}\tfrac{(i-3)^2}{2i^2-3i-1}\geq
  \tfrac{(N-2)(N^2-5N+10)}{2N(3N-5)}> 0,
  $$
 which implies $J_{N+1, N+1}>0$.
  Since we add a factor of $i-th$ row to $j-th$ row with $i<j$, all leading principal minors of matrices $K$ and $J$ are the same. Hence $K$ is positive definite. As $ E(c , c )=c  K$, one can infer the positive definiteness of $ E(c , c )$ and the proof is complete.
\end{proof}

In the upcoming lemma, we establish a valid inequality for ADMM that will be utilized in all the subsequent results presented in this section.

\begin{lemma}\label{Lemma0}
Let $f\in\mathcal{F}^A_{c_1}(\mathbb{R}^n)$, $g\in\mathcal{F}_{0}(\mathbb{R}^m)$ and $x^\star=0$, $z^\star=0$. Suppose that ADMM with the starting points $\lambda^0$ and $z^0$ generates $\{(x^k; z^k; \lambda^k)\}$. If $N\geq 4$ and $v\in\mathbb{R}^r$, then
 {\small{
\begin{align}\label{M.ineq}
\nonumber & N\langle \lambda^N, Ax^N+Bz^N\rangle-
\langle \lambda^N+tAx^N+tBz^{N-1}, Ax^N-v\rangle+
\langle \lambda^{0}+tAx^1+tBz^0, Ax^1-v\rangle+\\
\nonumber& \tfrac{1}{2t}\left\| \lambda^{0}-\lambda^\star\right\|^2-
\tfrac{1}{2t}\left\| \lambda^{N}-\lambda^\star\right\|^2+
 \tfrac{t}{2}\left\| z^0\right\|^2_B-
 t \left\langle Ax^1-Ax^2+(N+1)Ax^N+Bz^N, v \right\rangle-
\\
\nonumber &   t\sum_{k=3}^{N} \langle Ax^k, v\rangle+
\tfrac{t(N-1)}{2}\left\|v\right\|^2-
\tfrac{c_1}{2}\left\|x^1\right\|_A^2+\sum_{k=2}^{N}\tfrac{\alpha_k}{2}\left\|x^k\right\|^2_A+\sum_{k=2}^{N-1}\beta_k\langle Ax^k,Ax^{k+1}\rangle+
\\
\nonumber & tN\langle Bz^{N-1}, Ax^N-v\rangle+
t\langle Ax^N, Bz^{N} \rangle-
\tfrac{t(N-1)^2}{2}\left\| z^N-z^{N-1}\right\|_B^2-
\tfrac{tN^2}{2}\left\|Ax^{N}+Bz^{N}\right\|^2-\\
& t\left\|x^2\right\|_A^2+f(x^1)-f(x^N)+N\left(f(x^N)-f^\star+g(x^N)-g^\star\right)\geq 0,
 \end{align}
}}
where
{\small{
\begin{align*}
&\alpha_k=
\begin{cases}
  \left(4k-1\right)t-2\left(2k^2-3k+1\right)c_1, & {2\leq k\leq N-1},\\
  \left(4N+1\right)t-\left(2N^2-5N+3\right)c_1, & {k=N},
\end{cases}
\\
&\beta_k=\left(2k^2-k-1\right)c_1-2kt.
\end{align*}
}}
\end{lemma}
\begin{proof}
To establish the desired inequality, we demonstrate its validity by summing a series of valid inequalities. To simplify the notation, let  $f^k=f(x^k)$ and $g^k=g(z^k)$ for $k\in\{1, \dots, N\}$. Note that $b=0$ because  $x^\star=0,  z^\star=0$. By \eqref{interp} and \eqref{OPT}, we get the following inequality

 {\small{
\begin{align*}
&
\sum_{k=1}^{N-1}(k^2-1)\left(f^{k+1}-f^k+\left\langle \lambda^{k-1}+tAx^k+tBz^{k-1}, A(x^{k+1}-x^k) \right\rangle-\tfrac{c_1}{2}\left\|x^{k+1}-x^k\right\|_A^2\right)\\
 & +\sum_{k=1}^{N-1}(k^2-k)\left(f^k-f^{k+1}+\left\langle \lambda^{k}+tAx^{k+1}+tBz^{k}, A(x^k-x^{k+1}) \right\rangle-\tfrac{c_1}{2}\left\|x^{k+1}-x^k\right\|_A^2\right)\\
&  +\sum_{k=1}^{N}\left(f^k-f^\star+\left\langle \lambda^\star, Ax^k \right\rangle-\tfrac{c_1}{2}\left\|x^k\right\|_A^2\right)+
 \sum_{k=1}^{N-1}k^2\left(g^k-g^{k+1}+\left\langle \lambda^{k+1}, B(z^{k}-z^{k+1}) \right\rangle\right)
\\
 & +\sum_{k=1}^{N-1}(k^2+k)\left(g^{k+1}-g^k+\left\langle \lambda^k, B(z^{k+1}-z^k) \right\rangle\right)+
 \sum_{k=1}^{N}\left(g^k-g^\star+\left\langle \lambda^\star, Bz^k \right\rangle\right)\\
 &+\tfrac{t}{2}\left\| Ax^1+Bz^0-v\right\|^2\geq 0.
 \end{align*}
}}
As $\lambda^k = \lambda^{k-1} + tAx^k + tBz^k$, the inequality can be expressed as
 {\small{
\begin{align*}
&
 \sum_{k=1}^{N-1}(k^2-1)
 \left(\left\langle tAx^k+tBz^{k-1}, A(x^{k+1}-x^k) \right\rangle-\tfrac{c_1}{2}\left\|x^{k+1}-x^k\right\|_A^2\right)+\\
& \sum_{k=1}^{N-1}(k^2-1)
\left(\left\langle \lambda^{k}, Ax^{k+1} \right\rangle-\left\langle \lambda^{k-1}, Ax^{k} \right\rangle-\left\langle tAx^k+tBz^k, Ax^{k+1} \right\rangle\right)+\\
 & \sum_{k=1}^{N-1}(k^2-k)
 \left(\left\langle tAx^{k+1}+tBz^{k}, A(x^k-x^{k+1}) \right\rangle-\tfrac{c_1}{2}\left\|x^{k+1}-x^k\right\|_A^2\right)+\\
 & \sum_{k=1}^{N-1}(k^2-k)
 \left(\left\langle \lambda^{k-1}, Ax^{k} \right\rangle-\left\langle \lambda^{k}, Ax^{k+1} \right\rangle+\left\langle tAx^k+tBz^k, Ax^{k} \right\rangle\right)+\\
 &\sum_{k=1}^{N-1}(k^2+k)
 \left(\left\langle \lambda^k, Bz^{k+1} \right\rangle- \left\langle \lambda^{k-1}, Bz^k \right\rangle- \left\langle tAx^k+tBz^k, Bz^k \right\rangle\right)+\\
& \sum_{k=1}^{N-1}k^2
\Bigg(\left\langle \lambda^{k-1}, Bz^{k} \right\rangle-\left\langle \lambda^{k}, Bz^{k+1} \right\rangle+
\left\langle tAx^{k}+tBz^{k}+tAx^{k+1}+tBz^{k+1}, Bz^{k} \right\rangle-\\
&\ \ \ \ \left\langle tAx^{k+1}+tBz^{k+1}, Bz^{k+1} \right\rangle\Bigg)+
 \sum_{k=1}^{N}\left(\left\langle \lambda^\star, Ax^k+Bz^k \right\rangle-\tfrac{c_1}{2}\left\|x^k\right\|_A^2\right)+
\tfrac{t}{2}\left\| Bz^0\right\|^2+\\
 & \tfrac{t}{2}\left\| Ax^1-v\right\|^2+t\left\langle Ax^1-v, Bz^0\right\rangle+f^1-f^N+N(f^N-f^\star+g^N-g^\star)\geq 0.
 \end{align*}
}}
After performing some algebraic manipulations, we obtain
 {\small{
\begin{align*}
&N\langle \lambda^{N-1}, Ax^N+Bz^N\rangle-
\langle \lambda^{N-1}, Ax^N\rangle+\langle \lambda^{0}, Ax^1\rangle-
\sum_{k=0}^{N-1} \langle \lambda^k-\lambda^\star, Ax^{k+1}+Bz^{k+1}\rangle+\\
& \tfrac{t}{2}\left\| Ax^1-v\right\|^2+\tfrac{t}{2}\left\| Bz^0\right\|^2+t\left\langle Ax^1-v, Bz^0\right\rangle-
t(N^2-3N+1) \langle Ax^N, Bz^{N-1} \rangle-\\
& t\sum_{k=1}^{N-1}\left( (k-1)^2 \|Ax^k\|^2-(k^2-k)\langle Ax^k, Ax^{k+1}\rangle -(k^2-1)\langle Ax^{k+1}, Bz^{k-1}\rangle  \right)-\\
& t\sum_{k=1}^{N-1} \left( (k^2-k+1) \|Bz^k\|^2+(-k^2+k+1)\langle Ax^k, Bz^k\rangle-k^2\langle Bz^k, Bz^{k+1}\rangle\right)-\\
&t\sum_{k=2}^{N-1} \left( (2k^2-3k) \langle Ax^{k},  Bz^{k-1} \rangle\right)-t(N-1)^2\|Bz^{N}\|^2-t(N^2-3N+2)\|Ax^N\|^2-\\
&t(N-1)^2\langle Ax^N, Bz^{N} \rangle-
\sum_{k=1}^{N-1}\left((2k^2-k-1)\tfrac{c_1}{2}\left\|x^{k+1}-x^k\right\|_A^2+\tfrac{c_1}{2}\left\|x^{k+1}\right\|_A^2\right)-\\
 & \tfrac{c_1}{2}\left\|x^{1}\right\|_A^2+f^1-f^N+N(f^N-f^\star+g^N-g^\star)\geq 0.
 \end{align*}
}}
By using $\lambda^{N-1}=\lambda^N-tAx^N-tBz^N$ and
$$
2\langle \lambda^k-\lambda^\star,  Ax^{k+1}+Bz^{k+1}\rangle=\tfrac{1}{t}\| \lambda^{k+1}-\lambda^\star\|^2-\tfrac{1}{t}\| \lambda^{k}-\lambda^\star\|^2-
t\| Ax^{k+1}+Bz^{k+1}\|^2,
$$
we get
 {\small{
\begin{align*}
& N\langle \lambda^N, Ax^N+Bz^N\rangle-
\langle \lambda^N+tAx^N+tBz^{N-1}, Ax^N-v\rangle+
\langle \lambda^{0}+tAx^1+tBz^0, Ax^1-v\rangle\\
& +\tfrac{1}{2t}\left\| \lambda^{0}-\lambda^\star\right\|^2
-\tfrac{1}{2t}\left\| \lambda^{N}-\lambda^\star\right\|^2+
 \tfrac{t}{2}\left\| z^0\right\|_B^2-t \left\langle Ax^1-Ax^2+(N+1)Ax^N+Bz^N, v \right\rangle\\
& -t\sum_{k=3}^{N} \left\langle Ax^k, v\right\rangle
-\tfrac{t}{2} \sum_{k=2}^{N-1}\left\| (k-1)Bz^{k-1}-(k-1)Bz^{k}+kAx^{k}-(k+1)Ax^{k+1}+v \right\|^2\\
& +\tfrac{t(N-1)}{2}\left\|v\right\|^2-\frac{c_1}{2}\left\|x^1\right\|_A^2
-2t\left\|x^2\right\|_A^2
+\tfrac{1}{2}\sum_{k=2}^{N-1}\left(\left(4k-1\right)t-2\left(2k^2-3k+1\right)c_1\right)\left\|x^k\right\|^2_A\\
&+\sum_{k=2}^{N-1}\left(\left(2k^2-k-1\right)c_1-2kt\right)\langle Ax^k,Ax^{k+1}\rangle
+\left(\left(2N+\tfrac{1}{2}\right)t-\left(N^2-\tfrac{5}{2}N+\tfrac{3}{2}\right)c_1\right)\left\|x^N\right\|^2_A\\
&+tN\left\langle Bz^{N-1}, Ax^N-v\right\rangle+
t\left\langle Ax^N, Bz^{N}\right\rangle
-\tfrac{t\left(N-1\right)^2}{2}\left\| z^N-z^{N-1}\right\|_B^2
\\
& -\tfrac{tN^2}{2}\left\|Ax^{N}+Bz^{N}\right\|^2+f^1-f^N+N\left(f^N-f^\star+g^N-g^\star\right)\geq 0,
 \end{align*}
}}
which implies the desired inequality.

\end{proof}

We may now prove the main result of this section.

 \begin{theorem}\label{Th.M}
Let $f\in\mathcal{F}^A_{c_1}(\mathbb{R}^n)$ and $g\in\mathcal{F}_{0}(\mathbb{R}^m)$ with $c_1>0$. If $t\leq c_1$ and $N\geq 4$, then
\begin{align}\label{B2}
D(\lambda^\star)-D(\lambda^N)\leq \frac{\|\lambda^0-\lambda^\star\|^2+t^2\left\|z^0-z^\star\right\|_B^2}{4Nt}.
\end{align}
\end{theorem}

\begin{proof}
As discussed in Section \ref{sec_pep}, we may assume that $x^\star=0$ and $z^\star=0$. By \eqref{opt_lam_N}, we have $D(\lambda^N)=f(\hat x^{N})+g(z^N)+\left\langle \lambda ^N, A\hat x^{N}+Bz^N\right\rangle$ for some $\hat x^{N}$ with
 $-A^T\lambda^N\in\partial f(\hat x^{N})$. By employing \eqref{interp} and \eqref{OPT}, we obtain

  {\small{
\begin{align}\label{T3.ineq2}
\nonumber  & N\left(g(x^N)-g^\star+\langle \lambda^\star, Bz^N\rangle\right)+
 (N-1)\left(f(x^N)-f^\star+\langle \lambda^\star, Ax^N\rangle-\tfrac{c_1}{2}\left\|x^N\right\|_A^2\right)+\\
\nonumber   & \left(f(\hat x^{N})-f(x^1)+\left\langle \lambda^0+tAx^{1}+tBz^0, A\hat x^{N}-Ax^1\right\rangle-\tfrac{c_1}{2}\left\|\hat x^{N}-x^1\right\|_A^2\right)+\\
\nonumber  & (2N-2)\Bigg(  f(\hat x^{N})-f(x^N)+\left\langle \lambda^N-tBz^{N}+tBz^{N-1}, A\hat x^{N}-Ax^N\right\rangle-\\
&\ \ \ \tfrac{c_1}{2}\left\|\hat x^{N}-x^N\right\|_A^2 \Bigg)+
   \left(  f(\hat x^{N})-f^\star+\langle \lambda^\star, A\hat x^{N}\rangle-\tfrac{c_1}{2}\left\|\hat x^{N}\right\|_A^2 \right) \geq 0.
 \end{align}
}}
By substituting $v$ with $A\hat x^{N}$ in inequality \eqref{M.ineq} and summing it with  \eqref{T3.ineq2}, we get the following inequality after performing some algebraic manipulations

  {\small{
\begin{align}\label{T3.ineq3}
\nonumber & 2N\left( f(\hat x^{N})+g(x^N)+\left\langle \lambda^N, A\hat x^{N}+Bz^N\right\rangle-f^\star-g^\star\right)+
 \tfrac{1}{2t}\left\| \lambda^{0}-\lambda^\star\right\|^2+
 \tfrac{t}{2}\left\| z^0\right\|_B^2-\\
\nonumber & \tfrac{1}{2t}\left\|  \lambda^N-\lambda^\star+t(N-1)Ax^N+tA\hat x^{N}+tNBz^N \right\|^2-\\
&\tfrac{t}{2}\left\| (N-1)(Bz^{N-1}-Bz^N)+tAx^N-tA\hat x^{N} \right\|^2-\\
\nonumber & \tfrac{1}{2}\tr\left({E(t, c_1)}
\begin{pmatrix}
  Ax^1 & \dots & A\hat x^{N}
\end{pmatrix}^T
\begin{pmatrix}
  Ax^1 & \dots & A\hat x^{N}
\end{pmatrix}\right)\geq 0,
 \end{align}
}}
where the positive semidefinite matrix $E(t, c_1)$ is given in Lemma \ref{Lemma1}. As the inner product of positive semidefinite matrices is non-negative, inequality \eqref{T3.ineq3} implies that
$$
 2N\left(D(\lambda^\star)-D(\lambda^N)\right)\leq \tfrac{1}{2t}\left\| \lambda^{0}-\lambda^\star\right\|^2+ \tfrac{t}{2}\left\| z^0\right\|_B^2,
$$
and the proof is complete.

\end{proof}

 In comparison with Theorem \ref{TT1}, we could get a new convergence rate when only $f$ is strongly convex, i.e. $g$ does not need to be strongly convex. Also, the constant does not depend on $\lambda^1$. One important question concerning bound \eqref{B2} is its tightness, that is, if there is an optimization problem which attains the given convergence rate. It turns out that the bound \eqref{B2} is exact. The following example demonstrates  this point.

\begin{example}
Suppose that $c_1>0$, $N\geq 4$ and $t\in (0, c_1]$. Let  $f, g: \mathbb{R}\to\mathbb{R}$ be given as follows,
$$
f(x)=\tfrac{1}{2}|x|+\tfrac{c_1}{2}x^2, \ \
g(z)=\tfrac{1}{2}\max\{\tfrac{N-1}{N}(z-\tfrac{1}{2Nt})-\tfrac{1}{2Nt}, -z\}.
$$
Consider the optimization problem
\begin{align*}
\min_{(x, z)\in \mathbb{R}\times\mathbb{R}} f(x)+g(z),\\
\nonumber \st \ x+z=0,
\end{align*}
It is seen that  $A=B=I$ in this problem. Note that $(x^\star, z^\star)=(0, 0)$ with Lagrangian multiplier $\lambda^\star=\tfrac{1}{2}$ is an optimal solution and the optimal value is zero.  One can check that Algorithm \ref{ADMM} with initial point $\lambda^0=\tfrac{-1}{2}$ and $ z^0=0$ generates the following points,
\begin{align*}
 & x^k=0 & k\in\{1, \dots, N\} \\
 & z^k=\tfrac{1}{2Nt} & k\in\{1, \dots, N\} \\
  & \lambda^k=\tfrac{-1}{2}+\tfrac{k}{2N} & k\in\{1, \dots, N\}.
\end{align*}
At $\lambda^{N}$, we have $D(\lambda^N)=\tfrac{-1}{4Nt}=-\tfrac{\|\lambda^0-\lambda^\star\|^2+t^2\left\|z^0-z^\star\right\|_B^2}{4Nt}$, which shows the tightness of bound \eqref{B2}.
\end{example}

One important factor concerning dual-based methods that determines the efficiency of an algorithm is primal and dual feasibility (residual) convergence rates. In what follows, we study this subject under the setting of Theorem \ref{Th.M}. The next theorem gives a convergence rate in terms of  primal residual under the setting of Theorem \ref{Th.M}.

\begin{theorem}\label{Th.F}
Let $f\in\mathcal{F}_{c_1}^A(\mathbb{R}^n)$ and $g\in\mathcal{F}_{0}(\mathbb{R}^m)$ with $c_1>0$. If $t\leq c_1$ and $N\geq 4$, then
\begin{align}\label{B3_F}
\left\|Ax^N+Bz^N-b\right\|\leq \frac{\sqrt{\|\lambda^0-\lambda^\star\|^2+t^2\left\| z^0-z^\star\right\|_B^2}}{tN}.
\end{align}
\end{theorem}
\begin{proof}

The argument is similar to that used in the proof of Theorem \ref{Th.M}. By setting $v=Ax^{N}$ in \eqref{M.ineq}, one can infer the following inequality
 {\small{
\begin{align}\label{T4.ineq1}
\nonumber & N\left\langle \lambda^N, Ax^N+Bz^N\right\rangle+
\left\langle \lambda^{0}+tAx^1+tBz^0, Ax^1-Ax^N\right\rangle+
 \tfrac{1}{2t}\left\| \lambda^{0}-\lambda^\star\right\|^2+
 \tfrac{t}{2}\left\| z^0\right\|_B^2-\\
\nonumber & t \left\langle Ax^1-Ax^2, Ax^N \right\rangle
+\tfrac{t(N-1)}{2}\left\|Ax^N\right\|^2
-t\sum_{k=3}^{N} \left\langle Ax^k, Ax^N\right\rangle-\frac{c_1}{2}\left\|x^1\right\|_A^2-t\left\|x^2\right\|_A^2+\\
\nonumber &\sum_{k=2}^{N-1}\left(\left(2k-\tfrac{1}{2}\right)t-\left(2k^2-3k+1\right)c_1\right)\left\|x^k\right\|^2_A+
\left(\left(\tfrac{3}{2}N-\tfrac{3}{2}\right)t-\left(N^2-\tfrac{5}{2}N+\tfrac{3}{2}\right)c_1\right)\left\|x^N\right\|^2_A+\\
\nonumber  &\sum_{k=2}^{N-1}\left(\left(2k^2-k-1\right)c_1-2kt\right)\langle Ax^k,Ax^{k+1}\rangle-
\tfrac{t\left(N-1\right)^2}{2}\left\| z^N-z^{N-1}\right\|_B^2-
\\
& \tfrac{tN^2}{2}\left\|Ax^{N}+Bz^{N}\right\|^2 +f(x^1)-f(x^N)+N\left(f(x^N)-f^\star+g(x^N)-g^\star\right)\geq 0.
 \end{align}
}}
By employing \eqref{interp} and \eqref{OPT}, we have
 {\small{
\begin{align}\label{T4.ineq2}
\nonumber  &  N\left(f^\star-f(x^N)-\langle \lambda^N+Bz^{N-1}-Bz^N, Ax^N\rangle-\tfrac{c_1}{2}\left\|x^N\right\|_A^2\right)+\\
& \left(f(x^N)-f^1+\left\langle \lambda^0+tAx^{1}+tBz^0, Ax^N-Ax^1\right\rangle-\tfrac{c_1}{2}\left\|x^N-x^1\right\|_A^2\right)+\\
\nonumber  & N\left(g^\star-g(x^N)-\langle \lambda^N, Bz^N\rangle\right)\geq 0.
 \end{align}
}}
By summing \eqref{T4.ineq1} and \eqref{T4.ineq2}, we obtain
 {\small{
\begin{align}\label{T4.ineq3}
\nonumber& \tfrac{1}{2t}\left\| \lambda^{0}-\lambda^\star\right\|^2+ \tfrac{t}{2}\left\| z^0\right\|_B^2-
 \tfrac{t\left(N-1\right)^2}{2}\left\| z^{N-1}-z^{N}+\tfrac{N}{(N-1)^2} x^N\right\|_B^2-\\
 & \tfrac{tN^2}{2}\left\|Ax^{N}+Bz^{N}\right\|^2-\tfrac{1}{2}\tr\left({D(t, c_1)}
\begin{pmatrix}
  Ax^1 & \dots & Ax^{N}
\end{pmatrix}^T
\begin{pmatrix}
  Ax^1 & \dots & Ax^{N}
\end{pmatrix}\right)
\geq 0,
 \end{align}
}}
where the matrix $D(t, c_1)$ is as follows,
\[
  D(t, c_1)=\scalemath{1}{\left(\begin{array}{cccccccccc}
      2c_1 & 0 & 0  &  0       & \dots & 0  & 0 & \dots & 0 & t-c_1\\
      0 & \alpha_2 & \beta_2  &  0       & \dots & 0  & 0 & \dots & 0 & -t\\
      0 & \beta_2  & \alpha_3 & \beta_3  & \dots & 0  & 0 &  \dots & 0 & t\\
       \vdots &\vdots & \vdots & \vdots  & \vdots & \vdots & \vdots& \vdots & \vdots & \vdots\\
        0 & 0 & 0 & 0 & \dots & \alpha_k & \beta_k & \dots & 0 & t\\
        \vdots & \vdots & \vdots & \vdots  & \vdots & \vdots & \vdots& \vdots & \vdots & \vdots\\
         0 & 0 & 0  &  0       & \dots & 0  & 0 & \dots & \alpha_{N-1} & \beta_{N-1}\\
         t-c_1 & -t & t & t & \dots & t  & t &  \dots & \beta_{N-1} &  \alpha_{N}\\
\end{array}\right)},
\]
and
\begin{align*}
&\alpha_k=
\begin{cases}
  6c_1-5t, & {k=2}\\
  2\left(2k^2-3k+1\right)c_1-\left(4k-1\right)t, & {3\leq k\leq N-1},\\
  \left(2N^2-4N+4\right)c_1-\left(3N-5+\frac{N^2}{\left(N-1\right)^2}\right)t, & {k=N},\\
\end{cases}
\\
& \beta_k=2kt-\left(2k^2-k-1\right)c_1, \ \ \  {2\leq k\leq N-1}
\end{align*}
As the matrix $D(t, c_1)$ is positive semidefinite, see Appendix \ref{appen1}, inequality \eqref{T4.ineq3} implies that
$$
  \tfrac{tN^2}{2}\left\|Ax^{N}+Bz^{N}\right\|^2\leq \tfrac{1}{2t}\left\| \lambda^{0}-\lambda^\star\right\|^2+ \tfrac{t}{2}\left\| z^0\right\|_B^2,
$$
and the proof is complete.
\end{proof}

The following example shows the exactness of bound \eqref{B3_F}.

\begin{example}
Let $c_1>0$, $N\geq 4$ and $t\in (0, c_1]$. Consider functions   $f, g: \mathbb{R}\to\mathbb{R}$  given by the formulae follows,
\begin{align*}
 & f(x)=\tfrac{1}{2}|x|+\tfrac{ c_1}{2}x^2,\\
& g(z)=\max\{\left(\tfrac{1}{2}-\tfrac{1}{N}\right)\left(z-\tfrac{1}{Nt} \right),
\tfrac{1}{2}\left(\tfrac{1}{Nt}-z \right)\}.
\end{align*}

We formulate the following optimization problem,
\begin{align*}
\min_{(x, z)\in \mathbb{R}\times\mathbb{R}} f(x)+g(z),\\
\nonumber \st \ Ax+Bz=0,
\end{align*}
where $A=B=I$. One can verify that $(x^\star, z^\star)=(0, 0)$ with Lagrangian multiplier $\lambda^\star=\tfrac{1}{2}$ is an optimal solution.   Algorithm \ref{ADMM} with initial point $\lambda^0=\tfrac{-1}{2}$ and $ z^0=0$ generates the following points,
\begin{align*}
 & x^k=0 & k\in\{1, \dots, N\} \\
 & z^k=\tfrac{1}{Nt} & k\in\{1, \dots, N\} \\
  & \lambda^k=\tfrac{2k-N}{2N} & k\in\{1, \dots, N\}.
\end{align*}
At iteration $N$, we have $\|Ax^N+Bz^N\|=\tfrac{1}{tN}=\tfrac{\sqrt{\|\lambda^0-\lambda^\star\|^2+t^2\left\| z^0-z^\star\right\|_B^2}}{tN}$, which shows the tightness of bound \eqref{B3_F}.
\end{example}

In what follows, we study the convergence rate of ADMM in terms residual dual. To this end, we investigate the convergence rate of  $\{B\left(z^{k-1}-z^k\right)\}$ as $\left\|A^TB\left(z^{k-1}-z^k\right)\right\|\leq \|A\|\left\|z^{k-1}-z^k\right\|_B$. The next theorem provides a convergence rate for the aforementioned sequence.

\begin{theorem}\label{Th.FD}
Let $f\in\mathcal{F}^A_{c_1}(\mathbb{R}^n)$ and $g\in\mathcal{F}_{0}(\mathbb{R}^m)$ with $c_1>0$. If $t\leq c_1$ and $N\geq 4$, then
\begin{align}\label{B3}
\left\|z^N-z^{N-1}\right\|_B\leq \frac{\sqrt{\|\lambda^0-\lambda^\star\|^2+t^2\left\|  z^0-z^\star\right\|_B^2}}{(N-1)t}.
\end{align}
\end{theorem}
\begin{proof}
Similar to the proof of  Theorem \ref{Th.M}, by setting $v=Ax^{N}$ in \eqref{M.ineq} for $N-1$ iterations, one can infer the following inequality
\begin{align}\label{T5.ineq1}
\nonumber & (N-1)\langle \lambda^{N-1}, Ax^{N-1}+Bz^{N-1}\rangle+\tfrac{1}{2t}\| \lambda^{0}-\lambda^\star\|^2-
\tfrac{1}{2t}\| \lambda^{N-1}-\lambda^\star\|^2+\\
\nonumber& \tfrac{t}{2}\left\| z^0\right\|^2_B-\langle \lambda^{N-1}+tAx^{N-1}+tBz^{N-2}, Ax^{N-1}-Ax^N\rangle+
\tfrac{t(N-2)}{2}\|x^N\|_A^2 +\\
\nonumber&\langle \lambda^{0}+tAx^1+tBz^0, Ax^1-Ax^N\rangle-t \left\langle Ax^1-Ax^2+NAx^{N-1}+Bz^{N-1}, Ax^N \right\rangle\\
\nonumber&+\frac{1}{2}\sum_{k=2}^{N-2}\left(\left(4k-1\right)t-2\left(2k^2-3k+1\right)c_1\right)\left\|x^k\right\|^2_A+
t\langle Ax^{N-1}, Bz^{N-1} \rangle+\\
\nonumber&\sum_{k=2}^{N-2}\left(\left(2k^2-k-1\right)c_1-2kt\right)\langle Ax^k,Ax^{k+1}\rangle+t(N-1)\langle Bz^{N-2}, Ax^{N-1}-Ax^N\rangle\\
 \nonumber& +\frac{1}{2}\left(\left(4N-3\right)t-\left(2N^2-9N+10\right)c_1\right)\left\|x^{N-1}\right\|^2_A-t\left\|x^2\right\|_A^2-\tfrac{c_1}{2}\left\|x^1\right\|_A^2-\\
 \nonumber&\tfrac{t(N-2)^2}{2}\left\| z^{N-1}-z^{N-2}\right\|_B^2-
\tfrac{t(N-1)^2}{2}\|Ax^{N-1}+Bz^{N-1}\|^2-t\sum_{k=3}^{N-1} \langle Ax^k, Ax^N\rangle+\\
& f(x^1)-f(x^{N-1})+(N-1)(f(x^{N-1})-f^\star+g(x^{N-1})-g^\star)\geq 0.
 \end{align}
By using \eqref{interp} and \eqref{OPT}, we have
{\small{
\begin{align}\label{T5.ineq2}
\nonumber& (N^2-3N+2)\bigg(f(x^{N-1})-f(x^{N})+\left\langle \lambda^{N-1}+tAx^{N}+tBz^{N-1}, A\left(x^{N-1}-x^{N}\right) \right\rangle-\\
\nonumber&\tfrac{c_1}{2}\left\|x^{N}-x^{N-1}\right\|_A^2\bigg)+\bigg(f(x^{N})-f(x^1)+\bigg\langle \lambda^0+tAx^1+tBz^0, A\left(x^{N}-x^1\right) \bigg\rangle-\\
&\tfrac{c_1}{2}\|x^{N}-x^1\|_A^2\bigg)+N(N-1)\left(g(z^{N})-g(z^{N-1})+\left\langle \lambda^{N-1}, B\left(z^{N}-z^{N-1}\right) \right\rangle\right)+\\
\nonumber&(N^2-3N+1)\bigg(f(x^{N})-f(x^{N-1})+\bigg\langle \lambda^{N-1}-tBz^{N-1}+tBz^{N-2}, A\left(x^{N}-x^{N-1}\right) \bigg\rangle\\
\nonumber&-\tfrac{c_1}{2}\|x^{N}-x^{N-1}\|_A^2\bigg)+(N-1)\left(g^\star-g(z^{N})-\left\langle \lambda^{N-1}+tAx^N+tBz^N, Bz^N \right\rangle\right)+\\
\nonumber&(N-1)\bigg(f^\star-f(x^{N-1})-\bigg\langle \lambda^{N-1}-tBz^{N-1}+tBz^{N-2}, Ax^{N-1} \bigg\rangle-\tfrac{c_1}{2}\|x^{N-1}\|_A^2\bigg)+\\
\nonumber&(N-1)^2\left(g(z^{N-1})-g(z^{N})+\left\langle \lambda^{N-1}+tAx^N+Bz^N, B\left(z^{N-1}-z^{N}\right) \right\rangle\right)\geq0.
\end{align}
}}
By summing \eqref{T5.ineq1} and \eqref{T5.ineq2}, we obtain
{\small{
\begin{align*}
  &\tfrac{1}{2t}\left\|\lambda^0-\lambda^\star\right\|^2+\tfrac{t}{2}\left\|z^0\right\|_B^2-\tfrac{(N^2-1)t}{2}\left\|\tfrac{N}{N+1}Ax^N+Bz^N\right\|^2-\frac{t(N-1)^2}{2}\left\|z^N-z^{N-1}\right\|_B^2-\\
  &\tfrac{(N-2)^2t}{2}\bigg\| Bz^{N-2}-Bz^{N-1}+\tfrac{N-1}{N-2}Ax^{N-1}-\left(1-\tfrac{1}{(N-2)^2}\right)Ax^{N} \bigg\|^2-\\
    & \tfrac{1}{2}\tr\left({F(t, c_1)}
\begin{pmatrix}
  Ax^1 & \dots & Ax^{N}
\end{pmatrix}^T
\begin{pmatrix}
  Ax^1 & \dots & Ax^{N}
\end{pmatrix}\right)
\geq 0,
 \end{align*}
}}
where the matrix $F(t, c_1)$ is as follows,
\[
  F(t, c_1)=\scalemath{1}{\left(\begin{array}{cccccccccc}
      2c_1 & 0 & 0  &  0       & \dots & 0  & 0 & \dots & 0 & t-c_1\\
      0 & \alpha_2 & \beta_2  &  0       & \dots & 0  & 0 & \dots & 0 & -t\\
      0 & \beta_2  & \alpha_3 & \beta_3  & \dots & 0  & 0 &  \dots & 0 & t\\
       \vdots &\vdots & \vdots & \vdots  & \vdots & \vdots & \vdots& \vdots & \vdots & \vdots\\
        0 & 0 & 0 & 0 & \dots & \alpha_k & \beta_k & \dots & 0 & t\\
        \vdots & \vdots & \vdots & \vdots  & \vdots & \vdots & \vdots& \vdots & \vdots & \vdots\\
         0 & 0 & 0  &  0       & \dots & 0  & 0 & \dots & \alpha_{N-1} & \beta_{N-1}\\
         t-c_1 & -t & t & t & \dots & t  & t &  \dots & \beta_{N-1} &  \alpha_{N}\\
\end{array}\right)},
\]
and
\begin{align*}
&\alpha_k=
\begin{cases}
  6c_1-5t, & {k=2}\\
  2\left(2k^2-3k+1\right)c_1-\left(4k-1\right)t, & {3\leq k\leq N-1},\\
  \left(2N^2-6N+4\right)c_1-2\left(N+\frac{1}{(N-2)^2}-\frac{2}{N+1}-3\right)t, & {k=N},\\
\end{cases}
\\
& \beta_k=\begin{cases}
  2kt-\left(2k^2-k-1\right)c_1, \ \ \  {2\leq k\leq N-2},\\
  (N+\frac{1}{2-N}-1)t-(2N^2-6N+3)c_1, & {k=N-1},\\
\end{cases}
\end{align*}
The rest of the proof proceeds analogously to the proof of Theorem \ref{Th.F}.
\end{proof}

  The following example shows the tightness of this bound.

\begin{example}
Assume that $c_1>0$, $N\geq 4$ and $t\in (0, c_1]$ are given, and  $f, g: \mathbb{R}\to\mathbb{R}$  are defined by,
\begin{align*}
& f(x)=\tfrac{1}{2}\max\left\{-\tfrac{N+1}{N-1}x,x\right\}+\tfrac{c_1}{2}x^2, \\
& g(z)=\tfrac{1}{2}\max\left\{\tfrac{1}{t(N-1)}-z,\tfrac{N-3}{N-1}\left(z-\tfrac{1}{t(N-1)}\right)\right\}.
\end{align*}
Consider the optimization problem
\begin{align*}
\min_{(x, z)\in \mathbb{R}\times\mathbb{R}} f(x)+g(z),\\
\nonumber \st \ Ax+Bz=0.
\end{align*}
where $A=B=I$. The point $(x^\star, z^\star)=(0, 0)$ with Lagrangian multiplier $\lambda^\star=\tfrac{1}{2}$ is an optimal solution.  After performing $N$ iterations of Algorithm \ref{ADMM} with setting $\lambda^0=\tfrac{-1}{2}$ and $ z^0=0$, we have

\begin{align*}
 & x^k=0, \ \ \ \ \ \ \ \ \ \ k\in\{1, \dots, N\}, \\
 & z^k=\begin{cases}
         \tfrac{1}{t(N-1)}, & k\in\{1, \dots, N-1\}, \\
         0, & k=N,
       \end{cases}\\
  & \lambda^k=\begin{cases}
         \tfrac{2k+1-N}{2(N-1)}, & k\in\{1, \dots, N-1\}, \\
         \tfrac{1}{2}, & k=N.
       \end{cases}
\end{align*}
It can be seen that $\left\|A^TB\left(z^N-z^{N-1}\right)\right\|=\tfrac{1}{(N-1)t}=\tfrac{\sqrt{\|\lambda^0-\lambda^\star\|^2+t^2\left\| z^0-z^\star\right\|_B^2}}{(N-1)t}$, which shows that the bound is tight.
\end{example}

 Theorem \ref{Th.M} and \ref{Th.F} address the case that $f$ is strongly convex  relative to $\| .\|_A$ and $g$ is convex. Based on numerical results by solving performance estimation problems including \eqref{P3} we conjecture, under the assumptions of Theorem \ref{Th.M}, if $g$ is $c_2$-strongly convex  relative to $\| .\|_B$, Algorithm \ref{ADMM} enjoys the following convergence rates
\begin{align*}
& D(\lambda^\star)-D(\lambda^N)\leq \frac{\|\lambda^0-\lambda^\star\|^2+t^2\|z^0-z^\star\|_B^2}{4Nt+\tfrac{2c_1c_2}{c_1+c_2}}, \\
&\left\|Ax^N+Bz^N-b\right\|\leq \frac{\sqrt{\|\lambda^0-\lambda^\star\|^2+t^2\| z^0-z^\star\|_B^2}}{Nt+\tfrac{c_1c_2}{c_1+c_2}}.
\end{align*}
We have verified these conjectures numerically for many specific values of the parameters.
\section{Linear convergence of ADMM}\label{sec_L}
In this section we study the linear convergence of ADMM.
The linear convergence of ADMM has been addressed by some authors and some conditions for linear convergence have been proposed, see \cite{deng2016global, liu2018partial, han2022survey, han2018linear, hong2017linear, nishihara2015general, yuan2020discerning}. Two common types of assumptions  employed for proving the linear convergence of ADMM are error bound property and $L$-smoothness. To the best knowledge of authors, most scholars investigated the linear convergence of the sequence $\{(x^k, z^k, \lambda^k)\}$ to a saddle point and there is no result in terms of dual objective value for ADMM. In line with the previous section, we study the linear convergence in terms of dual objective value and we derive some formulas for linear convergence rate by using performance estimation. It is noteworthy to mention that the term "Q-linear convergence" is also employed to describe the linear convergence in the literature.

As mentioned earlier, error bound property is  used by scholars for establishing the linear convergence; see e.g. \cite{liu2018partial,  han2018linear, hong2017linear, pena2021linear, yuan2020discerning}. Let
\begin{align}\label{D_a}
D^a(\lambda):=\min f(x)+g(z)+\langle \lambda, Ax+Bz-b\rangle+\tfrac{a}{2}\|Ax+Bz-b\|^2,
\end{align}
stands for augmented dual objective for the given $a>0$ and $\Lambda^\star$ denotes the optimal solution set of the dual problem. Note that function $D^a$ is an $\tfrac{1}{a}$-smooth function on its domain without assuming strong convexity; see  \cite[Lemma 2.2]{hong2017linear}.

\begin{definition}
The function $D^a$ satisfies the error bound if we have
\begin{align}\label{EB_H}
d_{\Lambda^\star}(\lambda)\leq\tau \|\nabla D^a(\lambda)\|, \ \ \lambda\in\mathbb{R}^r,
\end{align}
for some $\tau>0$.
\end{definition}

Hong et al. \cite{hong2017linear} established the linear convergence by employing error bound property \eqref{EB_H}.

 Recently, some scholars established the linear convergence of gradient methods for $L$-smooth convex functions by replacing strong convexity with some mild conditions, see \cite{necoara2019linear, abbaszadehpeivasti2022conditions, bolte2017error} and references therein. Inspired by these results, we prove the linear convergence of ADMM by using the so-called P\L\ inequality. Concerning differentiability of dual objective, by \eqref{Con_C}, we have
\begin{align}\label{der_D}
 b-A\partial f^*(-A^T\lambda)-B\partial g^*(-B^T\lambda)\subseteq \partial\left(-D(\lambda)\right).
 \end{align}
Note that inclusion \eqref{der_D} holds as an equality under some mild conditions, see e.g. \cite[Chapter 3]{beck2017first}.

\begin{definition}\label{Def_PL}
The function $D$ is said to satisfy the P\L\ inequality  if there exists an $L_p>0$ such that for any $\lambda\in \mathbb{R}^r$ we have
\begin{align}\label{PL}
D(\lambda^\star)-D(\lambda) \leq \tfrac{1}{2L_p} \|\xi\|^2, \ \ \  \xi\in b-A\partial f^*(-A^T\lambda)-B\partial g^*(-B^T\lambda).
\end{align}
\end{definition}

Note that if $f$ and $g$ are strongly convex, then $-D$ is an $L$-smooth convex function with $L\leq \tfrac{\lambda_{\max}(A^TA)}{\mu_1}+\tfrac{\lambda_{\max}(B^TB)}{\mu_2}$. Under this setting,  we have $L_p\leq \tfrac{\lambda_{\max}(A^TA)}{\mu_1}+\tfrac{\lambda_{\max}(B^TB)}{\mu_2}$. This follows from the duality between smoothness and strong convexity and

{\small{
\begin{align*}
\left\| \nabla D(\lambda)-\nabla D(\nu)\right\|\leq \left\| \nabla f^*(-A^T\lambda)-\nabla f^*(-A^T\nu)\right\|_A+
\left\| \nabla g^*(-B^T\lambda)-\nabla g^*(-B^T\nu)\right\|_B\\
 \leq \tfrac{1}{\mu_1}\left\| A^T\lambda- A^T\nu\right\|_A+\tfrac{1}{\mu_2}\left\| B^T\lambda-B^T\nu\right\|_B
\leq  \left(\tfrac{\lambda_{\max}(A^TA)}{\mu_1}+\tfrac{\lambda_{\max}(B^TB)}{\mu_2}\right)\left\| \lambda-\nu\right\|.
\end{align*}
}}

 In the next proposition, we show that definitions \eqref{EB_H} and \eqref{PL} are equivalent.
 \begin{proposition}\label{Pro_El}
Let $L_a=\tfrac{1}{a}$ denote the Lipschitz constant of $\nabla D^a$, where $D^a$ is given in \eqref{D_a}.
\begin{enumerate}[i)]
  \item
  If $D^a$ satisfies the error bound \eqref{EB_H}, then $D$ satisfies the P\L\ inequality with $L_p=\tfrac{1}{L_a\tau^2}$.
  \item
   If $D$ satisfies  the P\L\ inequality, then $D^a$ satisfies the error bound \eqref{EB_H} with $\tau=\tfrac{L_p}{1+aL_p}$.
\end{enumerate}
\end{proposition}
\begin{proof}
First we prove $i)$. Suppose $\lambda\in \mathbb{R}^r$ and $\xi\in b-A\partial f^*(-A^T\lambda)-B\partial g^*(-B^T\lambda)$. By identity \eqref{Conj_inv}, we have $\xi=b-A\bar x-B\bar z$ for some $(\bar x, \bar z)\in\argmin f(x)+g(z)+\langle \lambda, Ax+Bz-b\rangle$. Due to the smoothness of $D^a$ and \eqref{EB_H}, we get
\begin{align}\label{PpP0}
D^a(\lambda^\star)-D^a(\nu) \leq \tfrac{L_a\tau^2}{2} \|\nabla D^a(\nu)\|^2, \ \ \nu\in\mathbb{R}^r.
\end{align}
 where $\lambda^\star\in \Lambda^\star$ with $d_{\Lambda^\star}=\|\nu-\lambda^\star\|$. Suppose that $\bar \nu=\lambda-a(A\bar x+B\bar z-b)$. As we assume strong duality, we have $D^a(\lambda^\star)=D(\lambda^\star)$. By the definitions of $\bar x, \bar y$, we get
 $$
(\bar x, \bar z)\in\argmin f(x)+g(z)+\langle \bar\nu, Ax+Bz-b\rangle+\tfrac{a}{2}\|Ax+Bz-b\|^2.
$$
 By \cite[Lemma 2.1]{hong2017linear}, we have $\nabla D^a(\bar\nu)=A\bar x+B\bar z-b$. This equality with \eqref{PpP0} imply
$$
D(\lambda^\star)-D(\lambda)\leq D^a(\lambda^\star)-D^a(\bar\nu)\leq  \tfrac{L_a\tau^2}{2} \|A\bar x+B\bar z-b\|^2,
$$
and the proof of $i)$ is complete. \\
Now we establish $ii)$. Let $\lambda$ be in the domain of $\nabla D^a$. By \cite[Lemma 2.1]{hong2017linear}, we have $\nabla D^a(\lambda)=A\bar x+B\bar z-b$ for some
$(\bar x, \bar z)\in\argmin  f(x)+g(z)+\langle \lambda, Ax+Bz-b\rangle+\tfrac{a}{2}\|Ax+Bz-b\|^2$, which implies that
\begin{align}\label{Por_1_PL}
0\in \partial f(\bar x)+A^T\left(\lambda+a(A\bar x+B\bar z-b)\right),  0\in \partial g(\bar z)+B^T\left(\lambda+a(A\bar x+B\bar z-b)\right).
\end{align}
 Supposing $\nu=\lambda+a(A\bar x+B\bar z-b)$. By \eqref{Por_1_PL}, one can infer that
$D(\nu)=f(\bar x)+g(\bar z)+\langle \nu, A\bar x+B\bar z-b\rangle$. In addition, \eqref{Conj_inv} implies that
$b-A\bar x-B\bar z\in b-A\partial f^*(-A^T\nu)-B\partial g^*(-B^T\nu)$. By the P\L\ inequality, we have
$$
\tfrac{1}{2L_p}\left\| A\bar x+B\bar z-b\right\|^2\geq D(\lambda^\star)-D(\nu)=D^a(\lambda^\star)-D^a(\lambda)-\tfrac{a}{2}\left\| A\bar x+B\bar z-b\right\|^2,
$$
where the equality follows from  $D(\nu)=D^a(\lambda)+\tfrac{a}{2}\left\| A\bar x+B\bar z-b\right\|^2$ and $D^a(\lambda^\star)=D(\lambda^\star)$. Hence,
$$
D^a(\lambda^\star)-D^a(\lambda) \leq \left(\tfrac{1}{2L_p}+\tfrac{a}{2} \right) \|\nabla D^a(\lambda)\|^2.
$$
This inequality says that  $D^a$ satisfies the P\L\ inequality. On the other hand, the P\L\ inequality implies the error bound with the same constant, see \cite{bolte2017error}, and the proof is complete.
\end{proof}

In what follows, we employ performance estimation to derive a linear convergence rate for ADMM in terms of dual objective when the P\L\ inequality holds. To this end, we compare the value of dual problem in two consecutive iterations, that is, $\tfrac{D(\lambda^\star)-D(\lambda^2)}{D(\lambda^\star)-D(\lambda^1)}$. The following optimization problem gives the worst-case convergence rate,
\begin{align}\label{Pp2}
\nonumber   \max & \ \tfrac{D(\lambda^\star)-D(\lambda^2)}{D(\lambda^\star)-D(\lambda^1)}\\
\st &  \  \{x^2, z^2, \lambda^2\}\ \textrm{is generated  by Algorithm \ref{ADMM} w.r.t.}\ f, g, A, B, b, \lambda^1, z^1  \\
\nonumber  & \ (x^\star, z^\star) \ \text{is an optimal solution and its Lagrangian multipliers is }\ \lambda^\star\\
\nonumber  & \ D \ \text{satisfies the P\L\ inequality}\\
\nonumber  & \ f\in \mathcal{F}_{c_1}^A(\mathbb{R}^n), g\in \mathcal{F}^B_{c_2}(\mathbb{R}^n)\\
\nonumber  & \ \lambda^1\in\mathbb{R}^r,  z^1\in\mathbb{R}^m, A\in\mathbb{R}^{r\times n}, B\in\mathbb{R}^{r\times m},b\in\mathbb{R}^{r}.
\end{align}

Analogous to our discussion in Section \ref{sec_pep}, we may assume without loss of generality
$b=0$, $\lambda^1=\begin{pmatrix} A & B \end{pmatrix} \begin{pmatrix} x^\dag \\   z^\dag \end{pmatrix}$ and $\lambda^\star=\begin{pmatrix} A & B \end{pmatrix} \begin{pmatrix} \bar x \\ \bar z \end{pmatrix}$ for some $\bar x,  x^\dag, \bar z,  z^\dag$. In addition, we assume that $\hat x^1\in\argmin f(x)+\langle \lambda^1, Ax\rangle$ and
$\hat x^2\in\argmin f(x)+\langle \lambda^2, Ax\rangle$. Hence,
$$
D(\lambda^1)=f(\hat x^1)+g(z^1)+\langle \lambda^1, A\hat x^1+Bz^1\rangle, \ \
D(\lambda^2)=f(\hat x^2)+g(z^2)+\langle \lambda^2, A\hat x^2+Bz^2\rangle,
$$
and
\begin{align}\label{inv_d}
& -A^T\lambda^1\in \partial f(\hat x^1), \ \  & -B^T\lambda^1\in \partial g(z^1),\\
\nonumber & -A^T\lambda^2\in \partial f(\hat x^2), \ \ & -B^T\lambda^2\in \partial g(z^2).
\end{align}
Moreover, by \eqref{inv_d} and \eqref{der_D}, we get
$$
-A\hat x^1-Bz^1\in\partial\left(-D(\lambda^1)\right), \ \ \ \ -A\hat x^2-Bz^2\in\partial\left(-D(\lambda^2)\right).
$$
On the other hand, $\lambda^2=\lambda^1+tAx^2+tBz^2$. Therefore, by using Theorem \ref{T1}, problem \eqref{Pp2} may be relaxed as follows,

{\small{
\begin{align}\label{Pp3}
\nonumber   \max & \ \frac{f^\star+g^\star-\hat f^2-g^2-\langle Ax^\dag+Bz^\dag+tAx^2+tBz^2, A\hat x^2+Bz^2\rangle}{f^\star+g^\star-\hat f^1-g^1-\langle Ax^\dag+Bz^\dag, A\hat x^1+Bz^1\rangle}\\
\nonumber \st &  \  \Big\{\left(\hat x^1, -A^TAx^\dag-A^TBz^\dag, \hat f^1\right), \left(x^2, -A^TAx^\dag-A^TBz^\dag-tA^TAx^2-tA^TBz^1, f^2\right), \\
\nonumber & \left(\hat x^2, -A^TAx^\dag-A^TBz^\dag-tA^TAx^2-tA^TBz^2, \hat f^2\right),
\left(0, -A^TA\bar x-A^TB\bar z, f^\star\right)\Big\} \\
\nonumber &  \textrm{satisfy interpolation constraints \eqref{interp}}  \\
\nonumber & \  \Big\{\left(z^1, -B^TAx^\dag-B^TBz^\dag, g^1\right), \left(z^2, -B^TAx^\dag-B^TBz^\dag-tB^TAx^2-tB^TBz^2, g^2\right), \\
\nonumber & \left(0, -B^TA\bar z-B^TB\bar z, g^\star\right)\Big\} \ \textrm{satisfy interpolation constraints \eqref{interp}}  \\
 & f^*+g^*- \hat f^1- g^1-\left\langle Ax^\dag+Bz^\dag, A\hat x^1+Bz^1 \right\rangle\leq\tfrac{1}{2L_p}\left\| A\hat x^1+Bz^1\right\|^2\\
\nonumber & f^*+g^*- \hat f^2- g^2-\left\langle Ax^\dag+Bz^\dag+tAx^2+tBz^2, A\hat x^2+Bz^2 \right\rangle\leq \tfrac{1}{2L_p}\left\| A\hat x^2+Bz^2\right\|^2\\
\nonumber & A\in\mathbb{R}^{r\times n}, B\in\mathbb{R}^{r\times m}.
\end{align}
}}
By deriving an upper bound for the optimal value of problem \eqref{Pp3} in the next theorem, we establish the linear convergence of ADMM in the presence of the P\L\ inequality.

 \begin{theorem}\label{Th.L}
Let $f\in\mathcal{F}_{c_1}^A(\mathbb{R}^n)$ and $g\in\mathcal{F}^B_{c_2}(\mathbb{R}^m)$ with $c_1, c_2>0$, and let $D$ satisfies the P\L\ inequality with $L_p$. Suppose that $t\leq\sqrt{c_1c_2}$.
\begin{enumerate}[(i)]
  \item
  If $c_1\geq c_2$, then
  \begin{align}\label{Li_r1}
   \frac{D(\lambda^\star)-D(\lambda^2)}{D(\lambda^\star)-D(\lambda^1)}\leq \frac{2c_1c_2-t^2}{2c_1c_2-t^2+L_pt\left(4c_1c_2-c_2t-2t^2\right)},
 \end{align}
 in particular, if $t=\sqrt{c_1c_2}$,
 \begin{align*}
   \frac{D(\lambda^\star)-D(\lambda^2)}{D(\lambda^\star)-D(\lambda^1)}\leq \frac{1}{1+L_p\left(2\sqrt{c_1c_2}-c_2\right)}.
 \end{align*}
  \item
   If $c_1< c_2$, then
  \begin{align}\label{Li_r2}
  & \frac{D(\lambda^\star)-D(\lambda^2)}{D(\lambda^\star)-D(\lambda^1)}\leq  \\
  \nonumber & \ \ \ \ \frac{4 c_2^2-2c_2 \sqrt{c_1c_2}-t^2}{4 c_2^2-2c_2 \sqrt{c_1c_2}-t^2+L_pt\left(8c_2^2+5c_2t-2\sqrt{c_1c_2}\left(1+\tfrac{t}{c_1}\right)\left(2c_2+t\right)\right)}.
  \end{align}
\end{enumerate}
\end{theorem}
\begin{proof}
The argument is based on weak duality. Indeed, by introducing suitable Lagrangian multipliers, we establish that the given convergence rates are  upper bounds for problem \eqref{Pp3}. First, we prove $(i)$. Assume that $\alpha$ denotes the right hand side of inequality \eqref{Li_r1}. As $2c_1c_2-t^2>0$ and $4c_1c_2-c_2t-2t^2>0$, we have $0<\alpha<1$. With some algebra, one can show that

 {\small{
\begin{align*}
& f^\star+g^\star-\hat f^2-g^2-\langle Ax^\dag+Bz^\dag+tAx^2+tBz^2, A\hat x^2+Bz^2\rangle-\\
& \alpha\left( f^\star+g^\star-\hat f^1-g^1-\langle Ax^\dag+Bz^\dag, A\hat x^1+Bz^1\rangle \right)+\\
& \alpha\left( \hat f^2-\hat f^1+\langle Ax^\dag+Bz^\dag, A\hat x^2-A\hat x^1\rangle-\tfrac{c_1}{2}\left\|\hat x^2-\hat x^1\right\|^2_A \right)+\\
& \alpha\left( f^2-\hat f^2+\langle Ax^\dag+Bz^\dag+tAx^2+tBz^2, Ax^2-A\hat x^2\rangle-\tfrac{c_1}{2}\left\|x^2-\hat x^2\right\|^2_A \right)+\\
& \alpha\left( \hat f^2-f^2+\langle Ax^\dag+Bz^\dag+tAx^2+tBz^1, A\hat x^2- Ax^2\rangle-\tfrac{c_1}{2}\left\|\hat x^2- x^2\right\|_A^2 \right)+\\
& \alpha\left( g^2-g^1+\langle Ax^\dag+Bz^\dag, Bz^2-Bz^1\rangle-\tfrac{c_2}{2}\left\|z^2-z^1\right\|_B^2 \right)+\\
& (1-\alpha)\Big(-f^\star-g^\star+\left\langle Ax^\dag+Bz^\dag+tAx^2+tBz^2, A\hat x^2+Bz^2 \right\rangle+ \hat f^2+  g^2+\\
& \  \tfrac{1}{2L_p}\left\| A\hat x^2+Bz^2\right\|^2 \Big)
= \tfrac{-c_1\alpha}{2}\left\| \hat x^1-\hat x^2\right\|^2_A-
\tfrac{c_2\alpha}{2}\left\| Bz^1-Bz^2+\tfrac{t}{c_2}Ax^2-\tfrac{t}{c_2}A\hat x^2\right\|^2-\\
& \ \ \alpha(c_1-\tfrac{t^2}{2c_2})\left\| Ax^2+\tfrac{tc_2}{2c_1c_2-t^2}Bz^2-\tfrac{tc_2-2c_1c_2+t^2}{t^2-2c_1c_2}A\hat x^2\right\|^2.
\end{align*}
}}
 Hence,  we get
\begin{align*}
& f^\star+g^\star-\hat f^2-g^2-\langle Ax^\dag+Bz^\dag+tAx^2+tBz^2, A\hat x^2+Bz^2\rangle\leq \\
& \ \ \alpha \left(f^\star+g^\star-\hat f^1-g^1-\langle Ax^\dag+Bz^\dag, A\hat x^1+Bz^1\rangle\right)
\end{align*}
for any feasible point of problem \eqref{Pp2} and the proof of the first part is complete. For $(ii)$, we proceed analogously to the proof of $(i)$, but with different Lagrange multipliers. Let $\beta$ denote the right hand side of inequality \eqref{Li_r2}, i.e.
$$
\beta=\frac{4 c_2^2-2c_2 \sqrt{c_1c_2}-t^2}{4 c_2^2-2c_2 \sqrt{c_1c_2}-t^2+L_pt\left(8c_2^2+5c_2t-2\sqrt{c_1c_2}\left(1+\tfrac{t}{c_1}\right)\left(2c_2+t\right)\right)}.
$$
It is seen that $0<\beta <1$. By doing some calculus, we have

 {\small{
\begin{align*}
& f^\star+g^\star-\hat f^2-g^2-\langle Ax^\dag+Bz^\dag+tAx^2+tBz^2, A\hat x^2+Bz^2\rangle-\\
& \beta\left( f^\star+g^\star-\hat f^1-g^1-\langle Ax^\dag+Bz^\dag, A\hat x^1+Bz^1\rangle \right)+\\
& \beta\left( \hat f^2-\hat f^1+\langle Ax^\dag+Bz^\dag, A\hat x^2-A\hat x^1\rangle-\tfrac{c_1}{2}\left\|\hat x^2-\hat x^1\right\|^2_A \right)+\\
& \sqrt{\tfrac{c_2}{c_1}}\beta
\left( f^2-\hat f^2+\langle Ax^\dag+Bz^\dag+tAx^2+tBz^2, Ax^2-A\hat x^2\rangle-\tfrac{c_1}{2}\left\|x^2-\hat x^2\right\|_A^2 \right)+\\
& \sqrt{\tfrac{c_2}{c_1}}\beta
\left( \hat f^2-f^2+\langle Ax^\dag+Bz^\dag+tAx^2+tBz^1, A\hat x^2- Ax^2\rangle-\tfrac{c_1}{2}\left\|\hat x^2- x^2\right\|_A^2 \right)+\\
& \sqrt{\tfrac{c_2}{c_1}}\beta
\left( g^2-g^1+\langle Ax^\dag+Bz^\dag, Bz^2-Bz^1\rangle-\tfrac{c_2}{2}\left\|z^2-z^1\right\|_B^2 \right)+\\
& \left(\sqrt{\tfrac{c_2}{c_1}}-1\right)\beta
\bigg( g^1-g^2+\langle Ax^\dag+Bz^\dag+tAx^2+tBz^2, Bz^1-Bz^2\rangle-\\
& \tfrac{c_2}{2}\left\|z^1-z^2\right\|_B^2 \bigg)+(1-\beta)\big(-f^\star-g^\star+\left\langle Ax^\dag+Bz^\dag+tAx^2+tBz^2, A\hat x^2+Bz^2 \right\rangle+\\
& \hat f^2+ g^2+\tfrac{1}{2L_p}\left\| A\hat x^2+Bz^2\right\|^2 \big)\\
& =
 -\tfrac{c_1\beta}{2}\left\| \hat x^1-\hat x^2\right\|^2_A-
\left(\sqrt{c_1c_2}\beta\right)\left\|A x^2-\left(1-\frac{t}{2 \sqrt{c_1c_2}}\right)A\hat x^2+ \frac{t}{2 \sqrt{c_1c_2}}Bz^1\right\|^2-\\
& \left(\frac{\beta-1}{2L_p}+\beta t \left(1-\frac{t}{4\sqrt{c_1c_2}}\right)\right)\bigg\| A\hat x^2-\left(\frac{\beta L_p \left(-2c_2\sqrt{c_1c_2}+4c_2^2-t^2\right)}{-\beta L_p t^2+2 \sqrt{c_1c_2} (2 \beta L_p t+\beta-1)}\right)^{\frac{1}{2}}Bz^1+\\
&\left(\frac{2\left(2\beta c_2L_p\left(t+c_2\right)+\sqrt{c_1c_2}\left(\beta-\beta L_p c_2-1\right)\right)}{-\beta L_p t^2+2 \sqrt{c_1c_2} (2 \beta L_p t+\beta-1)}\right)^{\frac{1}{2}}Bz^2\bigg\|^2.
\end{align*}
}}
The rest of the proof is similar to that of the former case.
\end{proof}

We computed the bounds in Theorem \ref{Th.L} by selecting suitable Lagrangian multipliers and solving the semidefinite formulation of problem \eqref{Pp3} by hand. The semidefinite formulation is formed analogous to problem \eqref{P4}.  Note that the optimal value of problem \eqref{Pp3} may be smaller than the bounds introduced in Theorem \ref{Th.L}. Indeed, our aim was to provide a concrete mathematical proof for the linear convergence rate.     However, the linear convergence rate factor is not necessarily tight.  Needless to say that the optimal value of problem \eqref{Pp3} also does not necessarily give the tight convergence factor as it is just a relaxation of problem  \eqref{Pp2}.

Recently the authors showed that the P\L\ inequality is necessary and sufficient conditions for the linear convergence of the gradient method with constant step lengths for  $L$-smooth  function; see\cite[Theorem 5]{abbaszadehpeivasti2022conditions}. In what follows, we  establish that the P\L\ inequality is a necessary  condition for the linear convergence of ADMM. Firstly, we present a lemma that is very useful for our proof.

\begin{lemma}\label{Lemma2}
Let $f\in\mathcal{F}_{c_1}^A(\mathbb{R}^n)$ and $g\in\mathcal{F}^B_{c_2}(\mathbb{R}^m)$. Consider Algorithm \ref{ADMM}. If  $(\hat x^1, z^1)\in\argmin f(x)+g(z)+\langle \lambda^1, Ax+Bz-b\rangle$, then
\begin{align}\label{inq_gr_2}
 \langle A\hat x^1+Bz^1-b, Ax^2+Bz^2-b\rangle \leq \left\|A\hat x^1+Bz^1-b\right\|^2.
\end{align}
\end{lemma}
\begin{proof}
 Without loss of generality we assume that $c_1=c_2=0$.  By optimality conditions, we have
 \begin{align*}
& f(\hat x^1)-\langle \lambda^1, Ax^2-A\hat x^1\rangle\leq f(x^2), \ \ \
g(z^1)-\langle \lambda^1, Bz^2-Bz^1\rangle\leq g(z^2),\\
& f(x^2)-\langle \lambda^1+t(Ax^2+Bz^1-b), A\hat x^1-Ax^2\rangle\leq f(\hat x^1),\\
& g(z^2)-\langle \lambda^1+t(Ax^2+Bz^2-b), Bz^1-Bz^2\rangle\leq g(z^1).
\end{align*}
By using these inequities, we get
{\small{
\begin{align*}
0\leq &
 \tfrac{1}{t}\left(f(x^2)-f(\hat x^1)+\left\langle \lambda^1,  Ax^2-A\hat x^1\right\rangle \right)+
 \tfrac{1}{t}\left( g(z^2)-g(z^1)+\left\langle \lambda^1, Bz^2-Bz^1\right\rangle\right)+
 \\
 & \tfrac{1}{t}\left(f(\hat x^1)-f(x^2)+\left\langle \lambda^1+t(Ax^2+Bz^1-b), A\hat x^1- Ax^2\right\rangle \right)+\\
 & \tfrac{1}{t}\left( g(z^1)-g(z^2)+\left\langle \lambda^1+t(Ax^2+Bz^2-b), Bz^1-Bz^2\right\rangle\right)\\
 =&\left\|A\hat x^1+Bz^1-b\right\|^2-\left\langle A\hat x^1+Bz^1-b, Ax^2+Bz^2-b\right\rangle-\tfrac{3}{4}\left\|B\left(z^1-z^2\right)\right\|^2-
 \\
&\left\|A\left(\hat x^1-x^2\right)+\tfrac{1}{2}B\left(z^1-z^2\right)\right\|^2.
\end{align*}
}}
Hence,  we have
$$
 \frac{\langle A\hat x^1+Bz^1-b, Ax^2+Bz^2-b\rangle}{\left\|A\hat x^1+Bz^1-b\right\|^2}\leq 1,
$$
which completes the proof.
\end{proof}

The next theorem establishes that the P\L\ inequality is a necessary condition for the linear convergence of ADMM.

\begin{theorem}\label{Th.PL}
Let $f\in\mathcal{F}^A_{c_1}(\mathbb{R}^n)$ and $g\in\mathcal{F}^B_{c_2}(\mathbb{R}^m)$. If Algorithm \ref{ADMM} is linearly convergent with respect to the dual objective value, then $D$ satisfies the P\L\ inequality.
\end{theorem}
\begin{proof}
  Consider $\lambda^1\in\mathbb{R}^r$ and $\xi\in b-A\partial f^*(-A^T\lambda^1)-B\partial g^*(-B^T\lambda^1)$. Hence,  $\xi=b-A\hat x^1-Bz^1$ for some $(\hat x^1, z^1)\in\argmin f(x)+g(z)+\langle \lambda, Ax+Bz-b\rangle$.
   If one sets $z^0=z^1$ and $\lambda^0=\lambda^1-t(A\hat x^1+Bz^1-b)$ in Algorithm \ref{ADMM}, the algorithm may generate $\lambda^1$.  As  Algorithm \ref{ADMM} is linearly convergent, there exist $\gamma\in [0, 1)$ with
   $$
   D(\lambda^\star)-D(\lambda^2)\leq \gamma \left( D(\lambda^\star)-D(\lambda^1) \right).
   $$
   So, we have
     $$
   (1-\gamma)\left( D(\lambda^\star)-D(\lambda^1)\right) \leq D(\lambda^2)-D(\lambda^1)\leq \left\langle A\hat x^1+Bz^1-b, \lambda^2-\lambda^1\right\rangle,
   $$
   where the last inequality follows  from the concavity of the function $D$. Since $\lambda^2-\lambda^1=t(Ax^2+Bz^2-b)$, Lemma \ref{Lemma2} implies that
   $$
   D(\lambda^\star)-D(\lambda^1)\leq \tfrac{t}{1-\gamma} \|\xi\|^2,
   $$
   so $D$ satisfies the P\L\ inequality.
\end{proof}

   Another assumption used for establishing  linear convergence is $L$-smoothness; see for example \cite{nishihara2015general, deng2016global, giselsson2016linear, davis2017faster}. Deng et al. \cite{deng2016global} show that the sequence $\{(x^k, z^k, \lambda^k)\}$  is convergent linearly to a saddle point under  Scenario 1 and 2 given in Table \ref{table1}.

\begin{table}[H]
\caption{Scenarios leading to linear convergence rates \label{table1}}
\centering
\begin{tabular}{l l l l l l}
\hline
 Scenario & Strong convexity & Lipschitz continuity & Full row rank  \\ \hline
  1 & $f, g$ & $\nabla f$           &  $A$ \\
  2 & $f, g$ & $\nabla f, \nabla g$ &  -  \\
  3 & $f$ & $\nabla f, \nabla g$ &  $B^T$  \\
    \hline
\end{tabular}
\end{table}

It is worth mentioning that Scenario 1 or Scenario 2  implies strong convexity of the dual objective function and therefore the P\L\ inequality is resulted, see \cite{abbaszadehpeivasti2022conditions}. Hence, Theorem \ref{Th.L} implies the linear convergence in terms of dual value under Scenario 1 or Scenario 2. Deng et al. \cite{deng2016global} studied the linear convergence under Scenario 3, but they just proved the linear convergence of the sequence $\{(x^k, Bz^k, \lambda^k)\}$. In the next section, we investigate the R-linear convergence without assuming $L$-smoothness of $f$. Indeed, we establish the R-linear convergence when $f$ is strongly convex, $g$ is $L$-smooth and $B$ has full row rank.

  Note that the P\L\ inequality does not imply necessarily Scenario 1 or Scenario 2. Indeed, consider the following optimization problem,
\begin{align*}
&\min \ f(x)+g(z),\\
&\st  \ x+z=0,
\\ &  \ \ \ \ \ \ \ x,z\in \mathbb{R}^n,
\end{align*}
where $f(x)=\tfrac{1}{2}\|x\|^2+\|x\|_1$ and $g(z)=\tfrac{1}{2}\|z\|^2+\|z\|_1$. With some algebra, one may show that $D(\lambda)=\sum_{i=1}^{n} h(\lambda_i)$ with
$$
h(s)=
\begin{cases}
-(s-1)^2, & s>1\\
  0, &   |s|\leq 1\\
  -(s+1)^2, & s<-1.
\end{cases}
$$
Hence, the P\L\ inequality holds for $L_p=\tfrac{1}{2}$  while neither $f$ nor $g$ is $L$-smooth.

As mentioned earlier the performance estimation problem including the P\L\ inequality at finite set of points is a relaxation for computing the worst-case convergence rate. Contrary to  Theorem \ref{Th.L}, we could not manage to prove the linear convergence of primal and dual residuals under the assumptions of Theorem \ref{Th.L} by employing performance estimation.

\section{R-linear convergence of ADMM}\label{Sec.R}
In this section, we study the R-linear convergence of ADMM. Recall that ADMM enjoys R-linear convergent in terms of dual objective value if
$$
D(\lambda^\star)-D(\lambda^N)\leq \rho\gamma^N,
$$
for some $\rho\geq 0$ and $\gamma\in [0, 1)$.

We investigate the R-linear convergence under the following scenarios:
\begin{itemize}
  \item (S1):
  $f\in\mathcal{F}_{c_1}^A(\mathbb{R}^n)$ is $L$-smooth with $c_1>0$ and $A$ has full row rank;
  \vspace{.2cm}
  \item (S2):
   $f\in\mathcal{F}_{c_1}^A(\mathbb{R}^n)$ with $c_1>0$, $g$ is $L$-smooth and $B$ has full row rank.
\end{itemize}

Our technique for proving the R-linear convergence is based on establishing the linear convergence of the sequence $\{V^k\}$ given by
\begin{align}\label{Lyapunov}
  V^k=\|\lambda^k-\lambda^\star\|^2+t^2\left\|z^k-z^\star\right\|_B^2.
\end{align}
Note that $V^k$ is called Lyapunov function for ADMM; see \cite{boyd2011}.

First we consider the case that $f$ is $L$-smooth and $c_1$-strongly convex relative to $A$. The following proposition establishes the linear convergence of $\{V^k\}$.
\begin{proposition}\label{RL_p1}
Let $f\in\mathcal{F}_{c_1}^A(\mathbb{R}^n)$ be $L$-smooth with $c_1>0$, $g\in\mathcal{F}_{0}(\mathbb{R}^m)$ and let $A$ has full row rank. If  $t< \sqrt{\tfrac{c_1 L}{\lambda_{\min}(AA^T)}}$, then
\begin{align}\label{RL_B1}
V^{k+1}\leq \left(1-\tfrac{2c_1 t}{c_1d+2c_1 t+ t^2} \right) V^k,
\end{align}
where $d=\tfrac{L}{\lambda_{\min}(AA^T)}$.
\end{proposition}
\begin{proof}
 We may assume without loss of generality that $x^\star, z^\star$ and $b$ are zero; see our discussion in Section \ref{sec_pep}. By optimality conditions, we have
 \begin{align*}
& \nabla f(x^{k+1})=-A^T\left(\lambda^{k}+tAx^{k+1}+tBz^k\right),
&& \eta^k=-B^T\lambda^{k+1},\\
& \nabla f(x^\star)=-A^T\lambda^\star,
&& \eta^\star=-B^T\lambda^\star,
\end{align*}
for some $\eta^k\in\partial g(z^{k+1})$ and $\eta^\star\in\partial g(z^\star)$.
Let  $\alpha=\tfrac{2t}{c_1^2d^2+2c_1dt^2-4c_1^2t^2+t^4}$.
By Theorem \ref{T1}, we get
{\footnotesize{
\begin{align*}
&
{\alpha\left(t^2 + c_1d\right)^2}
\left( f(x^{k+1})-f^\star+\left\langle \lambda^\star, Ax^{k+1}\right\rangle- \tfrac{1}{2L}\left\|A^T\left(\lambda^{k}+tAx^{k+1}+tBz^k-\lambda^\star\right)\right\|^2\right)+\\
&
2\alpha t^2{\left(c_1d+t^2\right)}
\left(f^\star-f(x^{k+1})- \tfrac{c_1}{2}\left\|x^{k+1} \right\|^2_A-\left\langle \lambda^{k}+tAx^{k+1}+tBz^k, Ax^{k+1}\right\rangle\right)+\\
 &
 2t\left( g(z^{k+1})-g^\star+\left\langle \lambda^\star, Bz^{k+1}\right\rangle\right)+
 2t\left( g^\star-g(z^{k+1})-\left\langle \lambda^{k+1}, Bz^{k+1}\right\rangle\right)+
\\
 &
 {\alpha\left(c_1^2d^2-t^4\right)}
  \Bigg( f^\star-f(x^{k+1})-\left\langle \lambda^{k}+tAx^{k+1}+tBz^k, Ax^{k+1}\right\rangle-\\
 & \tfrac{1}{2L}\left\|A^T\left(\lambda^{k}+tAx^{k+1}+tBz^k-\lambda^\star\right)\right\|^2\Bigg)
 \geq 0.
\end{align*}
}}
As  $\|A^T\lambda\|^2\geq \tfrac{L}{d}\|\lambda\|^2$ and $\lambda^{k+1}=\lambda^k+tAx^{k+1}+tBz^{k+1}$, we obtain the following inequality after performing some algebraic manipulations
{\footnotesize{
\begin{align*}
&  \left(1-\tfrac{2c t}{cd+2c t+ t^2} \right)\left( \left\|\lambda^k-\lambda^\star\right\|^2+t^2\left\|Bz^k\right\|^2\right)- \left(\left\|\lambda^{k+1}-\lambda^\star\right\|^2+t^2\left\|Bz^{k+1}\right\|^2\right)-\\
& {2\alpha c_1^2t}\left\| \lambda^k-\lambda^\star+\tfrac{t^2+2c_1t+c_1d}{2c_1}Ax^{k+1}+\tfrac{t^2+c_1d}{2c_1}Bz^{k} \right\|^2\geq 0.
\end{align*}
}}
The above inequality implies that
$$
V^{k+1}\leq \left(1-\tfrac{2c_1 t}{c_1d+2c_1 t+ t^2} \right)V^k,
$$
and the proof is complete.
\end{proof}

Note that one can improve bound \eqref{RL_B1} under the assumptions of Proposition \ref{RL_p1} and the $\mu$-strong convexity of $f$ by employing the following known inequality
\begin{align*}
\tfrac{1}{2(1-\tfrac{\mu}{L})}&\left(\tfrac{1}{L}\left\|\nabla f(x)-\nabla f(y)\right\|^2+\mu\left\|x-y\right\|^2-\tfrac{2\mu}{L}\left\langle \nabla f(x)-\nabla f(y),x-y\right\rangle\right)  \\
& \leq f(y)-f(x)-\left\langle \nabla f(x), y-x\right\rangle.
\end{align*}
Indeed, we employed the given inequality but we could not manage to obtain a closed form formula for the convergence rate. The next theorem establishes  the R-linear convergence of ADMM in terms of dual objective value under the assumptions of Proposition \ref{RL_p1}.

\begin{theorem}\label{Th.RL1}
Let $N\geq 4$ and let $A$ has full row rank. Suppose that $f\in\mathcal{F}_{c_1}(\mathbb{R}^n)$ is $L$-smooth with $c_1>0$ and $g\in\mathcal{F}_{0}(\mathbb{R}^m)$. If
$t<\min\{c_1, \sqrt{\tfrac{c_1 L}{\lambda_{\min}(AA^T)}}\}$,
 then
\begin{align*}
D(\lambda^\star)-D(\lambda^N)\leq \rho\left(1-\tfrac{2c_1 t}{c_1d+2c_1 t+ t^2}  \right)^{N},
\end{align*}
where $d=\tfrac{L}{\lambda_{\min}(AA^T)}$ and $\rho=\tfrac{V^0}{16t}\left(1-\tfrac{2c_1 t}{c_1d+2c_1 t+ t^2}  \right)^{-4}$.
\end{theorem}
\begin{proof}
By Theorem \ref{Th.M} and Proposition \ref{RL_p1}, one can infer the following inequalities,
\begin{align*}
D(\lambda^\star)-D(\lambda^N)& \leq \tfrac{V^{N-4}}{16t}\\
& \leq \tfrac{V^0}{16t}\left( 1-\tfrac{2c_1 t}{c_1d+2c_1 t+ t^2} \right)^{N-4},
\end{align*}
which shows the desired inequality.
\end{proof}

Nishihara et al. \cite{nishihara2015general} showed the R-linear convergence of ADMM in terms of $\{x^k, z^k, \lambda^k\}$ under the following conditions:
\begin{enumerate}[i)]
  \item
 The function  $f$ is $L$-smooth and   $\mu$-strong with $\mu>0$;
  \item
  The matrix $A$ is invertible and that $B$ has full column rank.
\end{enumerate}
In Theorem \ref{Th.RL1}, we obtain the R-linear convergence under weaker assumptions. Indeed, we replace condition $ii)$ with the matrix $A$ having full row rank.

In the sequel, we investigate the R-linear convergence under the hypotheses of scenario (S2). The next proposition shows the linear convergence of $\{V^k\}$.
\begin{proposition}\label{RL_p2}
Let $f\in\mathcal{F}_{c_1}^A(\mathbb{R}^n)$ with $c_1>0$ and let $g\in\mathcal{F}_{0}(\mathbb{R}^m)$ be  $L$-smooth. Suppose that $B$ has full row rank and $k\geq 1$. If  $t\leq \min\{\tfrac{c_1}{2}, \tfrac{L}{2\lambda_{\min}(BB^T)}\}$, then
\begin{align}\label{RL_B2}
V^{k+1}\leq \left(\tfrac{L}{L+t\lambda_{\min}(BB^T)} \right)^2 V^k.
\end{align}
\end{proposition}
\begin{proof}
 Analogous to the proof of Proposition \ref{RL_p1}, we assume that $x^\star=0$, $z^\star=0$ and $b=0$. Due to the optimality conditions, we have
 \begin{align*}
& \xi^{k+1}=-A^T\left(\lambda^{k}+tAx^{k+1}+tBz^k\right),
&&  \xi^\star=-A^T\lambda^\star,\\
& \nabla g(z^k)=-B^T\lambda^k, \ \ \ \ \ \ \
 \nabla g(z^{k+1})=-B^T\lambda^{k+1},
&& \nabla g(z^\star)=-B^T\lambda^\star,
\end{align*}
for some $\xi^{k+1}\in\partial f(x^{k+1})$ and $\xi^\star\in\partial f(x^\star)$.  Suppose that  $d=\tfrac{L}{\lambda_{\min}(BB^T)}$ and $\alpha=\tfrac{2dt}{d+t}$. By Theorem \ref{T1}, we obtain
{\small{
\begin{align*}
&
 \frac{\alpha\left(d^2+t^2\right)}{d^2-t^2}
 \left(f^\star-f(x^{k+1})-\left\langle \lambda^{k}+tAx^{k+1}+tBz^k, Ax^{k+1}\right\rangle-\tfrac{c_1}{2}\left\|x^{k+1} \right\|_A^2\right)+\\
 &
   \frac{\alpha\left(d^2+t^2\right)}{d^2-t^2}
   \left(f(x^{k+1})-f(x^\star)+\left\langle \lambda^\star, Ax^{k+1}\right\rangle-\tfrac{c_1}{2}\left\|x^{k+1} \right\|_A^2\right)+\\
 &
 \alpha
 \left(g(z^{k+1})-g^\star+\left\langle \lambda^\star, Bz^{k+1}\right\rangle-\tfrac{1}{2L}\left\|B^T\left(\lambda^\star-\lambda^{k+1} \right) \right\|^2\right)+\\
  &
   \alpha
    \left(g^\star-g(z^{k+1})-\left\langle \lambda^{k+1}, Bz^{k+1}\right\rangle-\tfrac{1}{2L}\left\|B^T\left(\lambda^\star-\lambda^{k+1} \right) \right\|^2\right)+\\
  &
   \alpha
    \left(g(z^{k})-g(z^{k+1})+\left\langle \lambda^{k+1}, Bz^{k}-Bz^{k+1}\right\rangle-\tfrac{1}{2L}\left\|B^T\left(\lambda^{k+1}-\lambda^{k} \right) \right\|^2\right)+\\
  &
   \alpha
    \left(g(z^{k+1})-g(z^{k})+\left\langle \lambda^{k}, Bz^{k+1}-Bz^{k}\right\rangle-\tfrac{1}{2L}\left\|B^T\left(\lambda^{k+1}-\lambda^{k} \right) \right\|^2\right)\geq 0.
\end{align*}
}}
By employing  $\|B^T\lambda\|^2\geq \tfrac{L}{d}\|\lambda\|^2$ and $\lambda^{k+1}=\lambda^k+tAx^{k+1}+tBz^{k+1}$,  the aforementioned inequality can be expressed as follows after some algebraic manipulation,
{\small{
\begin{align*}
&
  \tfrac{-\alpha^2}{4}\left\| \left(\frac{2t^2}{d^2-dt}\right)Ax^{k+1}+Bz^k-\left(1+\tfrac{t}{d}\right)Bz^{k+1}
  \right\|^2-\frac{2 t \left(d^2+t^2\right) \left(cd^2-d t (c+t)-t^3\right)}{\left(d^2-t^2\right)^2}\\
&
 \left\|Ax^{k+1}\right\|^2-\tfrac{\alpha^2}{4d^2}\left\| \lambda^k-\lambda^\star+\left(\frac{2 d^2-(d-t)^2}{d-t}\right)Ax^{k+1}+\left(d+t\right)Bz^{k+1} \right\|^2\\
&+\left(\tfrac{d}{d+t} \right)^2\left( \left\|\lambda^k-\lambda^\star\right\|^2+t^2\left\|Bz^k\right\|^2\right)- \left(\left\|\lambda^{k+1}-\lambda^\star\right\|^2+t^2\left\|Bz^{k+1}\right\|^2\right)\geq 0.
\end{align*}
}}
Hence, we have
$$
V^{k+1}\leq \left(\tfrac{d}{d+t} \right)^2 V^k,
$$
and the proof is complete.
\end{proof}

As the sequence $\{V^k\}$ is not increasing \cite[Convergence Proof]{boyd2011}, we have $V^1\leq V^0$. Thus,  by using Theorem \ref{Th.M} and Proposition \ref{RL_p2}, one can infer the following theorem.

\begin{theorem}\label{Th.RL2}
Let $f\in\mathcal{F}_{c_1}^A(\mathbb{R}^n)$  with $c_1>0$ and let $g\in\mathcal{F}_{0}(\mathbb{R}^m)$ be  $L$-smooth. Assume that $N\geq 5$ and $B$ has full row rank. If
$t<\min\{\tfrac{c_1}{2}, \tfrac{L}{2\lambda_{\min}(BB^T)} \}$, then
\begin{align}\label{RL_B2}
D(\lambda^\star)-D(\lambda^N)\leq \rho\left(\tfrac{L}{L+t\lambda_{\min}(BB^T)} \right)^{2N},
\end{align}
where $\rho=\tfrac{V^0}{16t}\left(\tfrac{L}{L+t\lambda_{\min}(BB^T)} \right)^{-10}$.
\end{theorem}

In the same line, one can infer the R-linear convergence in terms of primal and dual residuals under the assumptions of Theorem \ref{Th.RL1} and Theorem \ref{Th.RL2}. In this section, we proved the linear convergence of $\{V^k\}$ under two scenarios (S1) and (S2). By \eqref{Con_C}, it is readily seen that function $-D$ is strongly convex under the hypotheses of both scenarios (S1) and (S2). Therefore, both scenarios imply the  P\L\ inequality. One may wonder that if the  P\L\ inequality and the strong convexity of $f$ imply the linear of $\{V^k\}$. By using performance estimation, we could not establish such an implication.

As mentioned above, function $-D$  under both scenarios are $\mu$-strongly convex. Hence, the optimal solution set of the dual problem is unique and one can infer the R-linear convergence of $\lambda^N$ by using Theorem \ref{Th.RL1} (Theorem \ref{Th.RL2}) and the known inequality,
$$
\tfrac{\mu}{2}\left\|\lambda^N-\lambda^\star\right\|^2\leq D(\lambda^\star)-D(\lambda^N).
$$

\section*{Concluding remarks}
In this paper we developed performance estimation framework to handle dual-based methods. Thanks to this framework, we could obtain some tight convergence rates for ADMM. This framework may be exploited for the analysis of other variants of ADMM in the ergodic and non-ergodic sense. Moreover, similarly to \cite{kim2016optimized}, one can apply this framework for introducing and analyzing new accelerated ADMM variants. Moreover, most results hold for any arbitrary positive step length, $t$, but we managed to get closed form formulas for some interval of positive numbers.
%
%

\bibliographystyle{spmpsci}      
\bibliography{references}
\appendix
\section{Appendix}\label{appen1}

\begin{lemma}\label{Lemma2}
Let $N\geq 4$ and $t, c_1\in\mathbb{R}$. Let $D(t, c_1)$ be $N\times N$ symmetric matrix given in Theorem \ref{Th.F}. If $c_1>0$ is given, then
$$
[0, c_1]\subseteq \{t: D(t, c_1)\succeq 0\}.
$$
\end{lemma}

\begin{proof}
The argument proceeds in the same manner as in Lemma \ref{Lemma1}. Due to the convexity of  $\{t: D(t, c_1)\succeq 0\}$, is sufficient to establish the  positive semidefiniteness of $D(0, c_1)$ and $D(c_1, c_1)$. As $D(0, c_1)$ is diagonally dominant, it is positive semidefinite. Next, we proceed to demonstrate the positive definiteness of the matrix $K = D(1,1)$ by computing its leading principal minors. One can show that the claim holds for $N=4$. So we investigate $N\geq 5$. To accomplish this, we perform the following elementary row operations on matrix $D$:
\begin{enumerate}[i)]
  \item
   Add the second row to the third row;
   \item
    Add the second row to the last row;
  \item
  Add the third row to the forth row;
  \item
  For $i=4:N-2$
  \begin{itemize}
    \item
     Add $i-th$ row to $(i+1)-th$ row;
    \item
     Add $\tfrac{3-i}{2i^2-3i-1}$ times of $i-th$ row to the last row;
  \end{itemize}
  \item
  Add $\frac{2N^2-8N+9}{2N^2-7N+4}$ times of $(N-1)-th$ row to $N-th$ row.
\end{enumerate}
By executing these operations, we transform $K$ into an upper triangular matirx $J$ with diagonal
\[
J_{k, k}=\begin{cases}
  2, & {k=1}\\
  2k^2-3k-1, & {2\leq k\leq N-1}\\
  2N^2- 7N + 8 - \frac{N^2}{(N - 1)^2} -\frac{(2N^2-8N+9)^2}{2N^2-7N+4}-\sum_{i=4}^{N-2}\tfrac{(i-3)^2}{2i^2-3i-1}, & {k=N}. \end{cases}
\]
 It is seen all first $(N-1)$ diagonal elements of $J$ are positive. We show that $J_{N, N}$ is also positive.  By using inequality \eqref{Lem_inq_pos}, we get
\begin{align*}
   &2N^2- 7N + 8 - \frac{N^2}{(N - 1)^2} -\frac{(2N^2-8N+9)^2}{2N^2-7N+4}-\sum_{i=4}^{N-2}\frac{(i-3)^2}{2i^2-3i-1}\geq\\
   &2N^2- 7N + 8 - \tfrac{25}{16} -(2N^2-8N+9)-\tfrac{N-5}{2}-1+\tfrac{2}{N-3}\geq \tfrac{N}{2}-\tfrac{17}{16}>0,
\end{align*}
for $N\geq 5$, which implies $J_{N, N}>0$. Hence,  $ D(c_1, c_1)\succeq 0$ and the proof is complete.
\end{proof}
\end{document}